\documentclass{amsart}          %% dvi output 

\usepackage{array}
\usepackage{enumerate}
\usepackage[margin=2.5cm]{geometry}
\usepackage[active]{srcltx}
\usepackage{amsfonts,amssymb,amsmath,latexsym,upref}
\usepackage[english]{babel}
\usepackage[latin1]{inputenc}  %æøåÆØÅ
\usepackage[all]{xy}\xyoption{arc}\xyoption{poly}
\usepackage{color}
\newtheorem{lemma}[equation]{Lemma}
\newtheorem{prop}[equation]{Proposition}
\newtheorem{thm}[equation]{Theorem}
\newtheorem{cor}[equation]{Corollary}

\newtheorem{thmx}{Theorem}
\newtheorem*{corx}{Corollary}
\theoremstyle{definition}
\newtheorem{defn}[equation]{Definition}

\theoremstyle{definition}
\newtheorem{exmp}[equation]{Example}

\newtheorem{rmk}[equation]{Remark}

\numberwithin{equation}{section}

\newcommand{\Q}{\mathbf{Q}}

\newcommand{\F}{\mathbf{F}}

\newcommand{\func}[3]{\mbox{$#1 \colon #2 \to #3$}}
\newcommand{\qt}[2]{#1 \backslash #2}

\newcommand{\Ob}[1]{\mathrm{Ob}(#1)}

\newcommand{\cat}[3]{\mathcal{#1}_{#2}^{#3}}  
\newcommand{\catt}[3]{\widetilde{\mathcal{#1}}_{#2}^{#3}} %t = tilde
\newcommand{\cattc}[2]{\widetilde{\mathcal{#1}}_{#2}^{\mathrm{sfc}}} 
\newcommand{\m}{morphism}

\newcommand{\syl}[1]{Sylow $#1$-subgroup}
\newcommand{\we}{weighting}

\newcommand{\Euc}{Euler characteristic}
\newcommand{\rchi}{\widetilde{\chi}}

\newcommand{\colim}{\operatornamewithlimits{colim}}

\title{Homotopy equivalences between $p$-subgroup categories}

\author{Matthew Gelvin}
\email{mgelvin@gmail.com}
\urladdr{}

\author{Jesper M.~M\o ller}
\address{Institut for Matematiske Fag\\
  Universitetsparken 5\\
  DK--2100 K\o benhavn}
\email{moller@math.ku.dk}
\urladdr{htpp://www.math.ku.dk/~moller}

\thanks
{Supported by the Danish National Research Foundation through the Centre 
for Symmetry and Deformation (DNRF92)}
\usepackage[bookmarks=true,bookmarksopen=false]{hyperref}
\hypersetup{
   pdftitle = {},
   pdfauthor = {Jesper Michael Møller},
   pdfpagemode = {UseOutlines},
   pdfstartview = {FitH},
   pdfborder = {0 0 0},
   backref = {true},
   colorlinks = {true},
   urlcolor = {blue},
   citecolor = {blue},
   linkcolor = {blue},
   pdftoolbar = {true},
}


\begin{document}
\date{\today}
\maketitle
\tableofcontents

\begin{abstract}
  Let $\cat SG*$ be the Brown poset of nonidentity $p$-subgroups of
  the finite group $G$ ordered by inclusion. Results of Bouc and
  Quillen show that $\cat SG*$ is homotopy equivalent to its subposets
  $\cat SG{*+\mathrm{rad}}$ of nonidentity radical $p$-subgroups and
  $\cat SG{*+\mathrm{eab}}$ of nonidentity elementary abelian
  $p$-subgroups. In this note we extend these results for the Brown
  poset of $G$ to other categories of $p$-subgroups of $G$, including the
  $p$-fusion system of $G$.
\end{abstract}

\section{Introduction}
\label{sec:intro}

Let $p$ be a prime number and $G$ a finite group of order divisible by $p$. 

The \emph{Brown poset} $\cat SG*$ consists of the nonidentity $p$-subgroups of $G$; this can be viewed topologically as   the simplicial complex $|\cat SG*|$.  Brown showed in \cite{brown75} that consideration of the Euler characteristic $\chi(|\cat SG*|)$ of $\cat SG*$ leads to a sort of topological Sylow Theorem: If $|G|_p$ is the maximal power of $p$ dividing $|G|$, then $\chi(|\cat SG*|)\equiv 1$ modulo $|G|_p$.  Thus, $G$ gives rise to a combinatorial geometric object---the realization of a category---whose topology reflects some of the algebraic structure of $G$.

$\cat SG*$ is comparatively simple for a category, but there are other constructions we might consider to gain further understanding of $G$.  Dwyer's theory of homology decompositions \cite{dwyer97} shows that even recalling the natural $G$-action on $\cat SG*$ is enough to determine the $p$-homology of $G$, though not conversely.  On the other hand,  the \emph{centric linking system} $\cat LG{\mathrm{sfc}}$ of \cite{blo1}, whose objects are the  \emph{$p$-selfcentralzing} subgroups of $G$, is significantly more complicated than $\cat SG*$ but both determines and is determined by the $p$-completed classifying space of $G$.

In each of these examples, the topological data is overdetermined by the category: We could have obtained the same result with a smaller collections of $p$-subgroups of $G$.  Quillen showed \cite{quillen78} that  the realization of $\cat SG*$ has the same homotopy type (and hence the same Euler characteristic) as the full subcategory $\cat SG{*+\mathrm{eab}}$ of nontrivial elementary abelian subgroups of $G$, and Bouc lated proved \cite{bouc84a} the dual result that the full subcategory $\cat SG{*+\mathrm{rad}}$ of $G$-radical $p$-subgroups is also homotopy equivalent.  Dwyer's notion of an \emph{ample} collection of subgroups is precisely the requirement that the $p$-homology of $G$ can be recovered from the resulting sub-$G$-poset; both elementary abelian and $G$-radical subgroups form such a collection.  And finally, the inclusion of the full subcategory of $G$-radical subgroups $\cat LG{\mathrm{sfc}+\mathrm{rad}}\subseteq\cat LG{\mathrm{sfc}}$ was shown to be a homotopy equivalence in \cite{bcglo01}.

In this paper we are interested in exploring the homotopy type of several such \emph{$p$-subgroup categories}:  Brown posets $\cat SG*$, transporter systems $\cat TG*$, linking systems $\cat LG*$, orbit systems $\cat OG{}$, and the ambient-group free abstractions of these: Frobenius categories (or fusion systems) $\cat F{}{}$ and exterior quotients of Frobenius categories (the fusion-theoretic analogue of an orbit category) $\catt F{}{}$.  See {\bf \S2} for definitions of these and their relationships to one another.  More precisely, we are interested in identifying certain classes of subgroups that control the homotopy type of each of these $p$-subgroup categories, in the sense of the main results  of {\bf \S\S6-10}:

   \begin{thmx}\label{thm:summary}
     The following inclusion functors are homotopy equivalences.
     \begin{enumerate}[(a)]
     \item  $\cat SG{*+\mathrm{eab}} \hookrightarrow \cat SG*$,
        $\cat SG{*+\mathrm{rad}} \hookrightarrow \cat SG*$, 
       $\cat SG{\mathrm{sfc}+\mathrm{rad}} \hookrightarrow \cat
       SG{\mathrm{sfc}}$
       \label{thm:summarya}

     \item  $\cat TG{*+\mathrm{eab}} \hookrightarrow \cat TG*$,
       $\cat TG{*+\mathrm{rad}} \hookrightarrow \cat TG*$, 
       $\cat TG{\mathrm{sfc}+\mathrm{rad}} \hookrightarrow \cat
       TG{\mathrm{sfc}}$

     \item  $\cat F{}{*+\mathrm{eab}} \hookrightarrow \cat F{}*$
       \label{thm:summaryb} 
     \item  $\cat OG{\mathrm{rad}} \hookrightarrow \cat OG{}$, 
       $\cat OG{*+\mathrm{rad}} \hookrightarrow \cat OG{*}$, 
       $\cat OG{\mathrm{sfc}+\mathrm{rad}} \hookrightarrow \cat
       OG{\mathrm{sfc}}$ \label{thm:summaryc}
     \item $\catt F{}{*+\mathrm{eab}} \hookrightarrow \catt F{}*$,
       $\catt F{}{\mathrm{sfc}+\mathrm{rad}} \hookrightarrow \catt
       F{}{\mathrm{sfc}}$
       \label{thm:summaryd}  
     \item $\cat L{G}{*+\mathrm{eab}} \hookrightarrow \cat L{G}*$,
       $\cat L{}{\mathrm{sfc}+\mathrm{rad}} \hookrightarrow \cat
       L{}{\mathrm{sfc}}$
       \label{thm:summarye} 
     \end{enumerate}
   \end{thmx}
   Here and for the rest of the paper, a functor of categories a \emph{homotopy equivalence} if the induced map of realizations is a homotopy equivalence.
   
   As indicated above, some of these results are well known in the literature, while others are new and provide new insight into the topological relationships between these categories.  For instance, Theorem~\ref{thm:summary}(\ref{thm:summaryb}) and(\ref{thm:summaryd}) together with the equality of categories $\cat F{}{*+\mathrm{eab}}=\catt F{}{*+\mathrm{eab}}$ implies the unexpected
   
    \begin{corx}\label{cor:FtotildeF}
   The quotient functor $\cat F{}* \to \catt F{}*$ is a homotopy
   equivalence.
  \end{corx}
  
    More generally, we find it curious that  the combinatorics of the Frobenius
  category $\cat FG*$, which is generally thought of as simply an organizing framework for the $p$-local data of $G$, can identify the elementary abelian
  $p$-subgroups of $G$.   Similarly, the orbit category $\cat OG{}$ is able to identify
  the $G$-radical and the cyclic subgroups, and the exterior
  quotient $\catt F{}{\mathrm{sfc}}$ of an abstract Frobenius category
  $\cat F{}{}$ is able to identify the $\cat F{}{}$-radical subgroups.  We take this as evidence of a general theme that $p$-subgroup categories  encode group structure in unexpected
  ways.
  
To understand the manner in which the shapes of our $p$-subgroup categories determine certain group-theoretic data,  we must discuss our method of proof for Theorem~\ref{thm:summary}.  Indeed, for us the method is at least as interesting as the final result. There are two interwoven threads to this story:  In one, we make use of Leinster's theory of Euler characteristics for a general category \cite{leinster08} to identify a class of subgroups that is likely to control homotopy of the $p$-subgroup category; in the other, we make use of a special case of Quillen's celebrated Theorem~A on homotopy equivalences of categories \cite{quillen73} to prove that our proposed class of subgroup actually does control the homotopy type.  We now summarize these points.

Consider first the special case of an inclusion of posets $\iota:\cat P{}{}\subseteq\cat Q{}{}$.  Here, Quillen's Theorem~A says that $\cat P{}{}$ is homotopy equivalent to $\cat Q{}{}$ if every slice or coslice category of $\iota$ is contractible.  In other words, it suffices to show that for every $q\in\cat Q{}{}$, we have $\cat P{}{}/q:=\{p\in\cat P{}{}|p\leq q\}$ is contractible; or dually that for every $q\in\cat Q{}{}$, $q/\cat P{}{}:=\{p\in\cat P{}{}|q\leq p\}$ is contractible.  Rephrased slightly: In order for $\iota$ to be a homotopy equivalence, it is necessary that every object $q\in\cat Q{}{}$ whose  \emph{proper} slice category $\cat P{}{}//q:=\{p\in\cat P{}{}|p\lneq q\}$ is \emph{noncontractible} be an object of $\cat P{}{}$, or that the dual criterion hold.  The first case led to Quillen's result $\cat SG{*+\mathrm{eab}}\simeq\cat SG*$, and the second to Bouc's homotopy equivalence $\cat SG{*+\mathrm{rad}}\simeq\cat SG*$.

When we generalize to our $p$-subgroup categories, it's important to note that we are not generalizing very far:  All of the categories $\cat C{}{}$ we consider in this paper are finite \emph{EI-categories}, so that every endomorphism of every object of $\cat C{}{}$ is an isomorphism.  For us then, the role of proper slice category $\cat C{}{}//x$ of $x\in\cat C{}{}$ is filled by the category of \emph{nonisomorphisms} with target $x$; the slice category is dually the category of nonisomorphisms with source $x$.  See {\bf \S3} for precise definitions.  We then have our main technical tool for showing that a class of subgroups controls homotopy, appearing as Theorem~\ref{lemma:boucEI2}:

  \begin{thmx}\label{thm:IntroBouc}
    The homotopy type of a finite EI-category $\cat C{}{}$ is controlled by the set of objects whose proper slice categories are noncontractible, or dually those objects whose proper coslice categories are noncontractible. 
  \end{thmx}
  
  In other words, if all you care about is the homotopy type of $\cat C{}{}$, you might as well throw away all objects $x$ such that $\cat C{}{}//x\simeq*$.  This material is covered in {\bf \S4}, and the key technical results needed for implementation of Theorem~\ref{thm:IntroBouc} is collected in {\bf\S5}.

Of course, this is only helpful if we have a proposed class of subgroups that might control homotopy of the $p$-subgroup category.  In order to direct our search we turn to Leinster's Euler characteristics for categories.  This is a generalization of the Euler characteristic of a poset or a space, which relies on the notions of \emph{weightings} and \emph{coweightings} for a category $\cat C{}{}$.  Roughly speaking, these are functions $k^\bullet_{\cat C{}{}},k_\bullet^{\cat C{}{}}:\mathrm{Ob}(\cat C{}{})\to{\mathbf Q}$ that serve as right and left ``inverses'' to the generalized incidence matrix, which records the number of morphisms between any two object of $\cat C{}{}$.  If both a weighting and a coweighting for $\cat C{}{}$ exist, as is always the case for finite EI-categories, then the \emph{Euler characteristic} $\chi(\cat C{}{})$ of $\cat C{}{}$ is the common sum of the values of either function.  The connection with Theorem~\ref{thm:IntroBouc} comes from Theorem~\ref{lemma:underwt}, whose key point is the following

\begin{thmx}
Let $\cat C{}{}$ be a finite EI-category.  There is a weighting for $\cat C{}{}$ that on an object $x$ takes a value proportional to the Euler characteristic of the proper coslice category of $x$ minus 1 (the \emph{reduced} Euler characteristic $\rchi$).  Dually, there is a coweighting whose value on $x$ is proportional to the reduced Euler characteristic of the proper slice category of $x$.  Thus there are constants $\kappa^x$ and $\kappa_x$ such that
\[
k^x_{\cat C{}{}}=\kappa^x\cdot\rchi(x//\cat C{}{})\qquad{and}\qquad
k_x^{\cat C{}{}}=\kappa_x\cdot\rchi(\cat C{}{}//x).
\]
\end{thmx}

In particular, the weighting is concentrated on the objects whose proper coslice categories have nonzero reduced Euler characteristic, and are thus necessarily noncontractible; dually for the coweighting.  This suggests (but does not prove!) that the classes of subgroups we consider should be those with nonzero (co)weightings in our $p$-subgroup categories, which were computed in \cite{jmm_mwj:2010}.  With our Euler characteristic calculations in hand, we conclude by applying Theorem~\ref{thm:IntroBouc}.

\section{$p$-subgroup categories}
\label{sec:subgrpcats}
This section contains precise definitions of the $p$-subgroup
categories occurring in this paper. By convention, maps act on
elements from the right, and composition of \m s is written in
diagrammatic order. Likewise, functors act on categories from the
right.

If $a$ and $b$ are objects in a category $\cat C{}{}$, we write $\cat C{}{}(a,b)$ for the set of $\cat C{}{}$-\m s with domain
  $a$ and codomain $b$, and $\cat C{}{}(a)$ is the monoid of $\cat C{}{}$-endo\m s of $a$.  All categories considered in this paper are EI-categories, so for us $\cat C{}{}(a)$ is in fact a group.

Fix a finite group $G$.  The most fundamental $p$-subgroup category  we consider is the poset $\cat SG{}$  of all $p$-subgroups of $G$, ordered
by inclusion. In other words, $\cat SG{}$ is the category whose
objects are all $p$-subgroups of $G$ with one \m\ $H \to K$ whenever
$H \leq K$ and no \m s otherwise.  

$\cat SG{}$ forms the backbone for all of the $p$-subgroup categories we consider. With our finite group $G$ still in mind, we consider the following categories, all of which have as objects the $p$-subgroups of $G$:

\begin{description}
\item[$\cat TG{}$]  The transporter category of  $G$; morphisms are elements of $G$ conjugating one subgroup into another.
\item[$\cat FG{}$] The Frobenius, or $p$-fusion, category of $G$
\cite{puig09,blo03}; morphisms are group-maps between subgroups induced by $G$-conjugation
\item[$\cat LG{}$] The linking category of all $p$-subgroups
  $G$; an intermediary between $\cat TG{}$ and $\cat FG{}$, thought of as  killing the $p'$-part of the kernel of the natural map $\cat TG{}\to\cat FG{}$. \cite{blo1}\footnote{Note that in much of the literature, only a full subcategory of $\cat LG{}$ is considered, the \emph{centric} linking system $\cat LG{\mathrm{sfc}}$. This full subcategory has much better homotopical properties than the full linking category, and we will concentrate on it once the notion of $\cat F{}{}$-selfcentralizing, or $\cat F{}{}$-centric, subgroup is recalled below.}
\item[$\cat OG{}$] The $p$-orbit category of  $G$; morphisms are $G$-maps between transitive $G$-sets with $p$-group isotropy.
\item[$\catt FG{}$]  The exterior quotient of the
  Frobenius category $\cat FG{}$; a fusion-theoretic analogue of $\cat OG{}$. \cite[1.3, 4.8]{puig09}
\end{description}

More explicitly, for any  $p$-subgroups $H$ and $K$ of $G$, the morphisms of the above categories are given by:
\begin{alignat*}{3}
  &\cat TG{}(H,K) = N_G(H,K) & &\qquad & 
  &\cat LG{}(H,K) = \qt{O^pC_G(H)}{N_G(H,K)} \\
  &\cat FG{}(H,K) = \qt{C_G(H)}{N_G(H,K)} & &\qquad & 
  &\cat OG{}(H,K) =  N_G(H,K)/K  \\
  &\catt FG{}(H,K) = \qt{C_G(H)}{N_G(H,K)}/K
\end{alignat*}
where $N_G(H,K)$ is the \emph{transporter set} $\{g \in G \mid H^g \leq K \}$ and $O^pL$ denotes the minimal normal $p$-power index subgroup of $L$.  Composition in these categories is induced by the group
multiplication of $G$.  

%The \m s in $\cat FG{}(H,K)$ are restrictions
%to $H$ of inner auto\m s of $G$, $\cat FG{}(H,K) =
%\mathrm{Hom}_G(H,K)$, \m s in $\mathcal{O}_G(H,K)$ are right $G$-maps
%$\qt HG \to \qt KG$\footnote{Define a product $\qt HG \times
%  N_G(H,K)/K \to \qt KG$ by $Hg \cdot xK = Kx^{-1}g$. This product is
%  well-defined because $Hhg \cdot xkK = Kk^{-1}x^{-1}hg = Kx^{-1}hg =
%  Kh^xx^{-1}g = Kx^{-1}g = Hg \cdot xK$ when $h \in H$, $x \in
%  N_G(H,K)$, $k \in K$. Right multiplikation with $xK$ is clearly a
%  right $G$-map $\qt HG \to \qt KG$. If $y \in N_G(K,L)$ then $xy \in
%  N_G(H,L)$ and $(Hg \cdot xK) \cdot yL = Kx^{-1}g \cdot yL =
%  Ly^{-1}x^{-1}g = L(xy)^{-1}g = Hg \cdot (xyL)$.}, and \m s in $\catt
%FG{}(H,K)$ are $K$-conjugacy classes of restrictions to $H$ of inner
%auto\m s of $G$, $\catt FG{}(H,K) = \cat F G{}(H,K)/\mathrm{Inn}(K) =
%\mathrm{Hom}_G(H,K)/\mathrm{Inn}(K)$.  

If $H\leq G$ is a $p$-subgroup, the auto\m\ groups in these
categories of $G$ are given by
%%%%%%%%%%%%
% \footnote{For $xK \in N_G(H,K)/K$, define a right
%   $G$-map $\varphi_{xK}:H\backslash G \to K\backslash G:Hg\mapsto
%   Kx^{-1}g$.  Clearly $\varphi_{xK}$ is independent of choice of
%   representative $x\in xK$, and $(Hh)\varphi_{xK}=Kx^{-1}h=Kh^x
%   x^{-1}=Kx^{-1}$ as $x\in N_G(H,K)$.  It is clear from the definition
%   that $\varphi_{xK}$ is $G$-equivariant. If $x \in N_G(H,K)$ and $y
%   \in N_G(K,L)$ then
%   $(Hg)\varphi_{xK}\varphi_{yL}=(Kx^{-1}g)\varphi_{yL}=(Ly^{-1}x^{-1}g)=(Hg)\varphi_{xyL}$.}
\[
\begin{array}{ccccc}
  \mathcal{S}_G(H)=1,&&\mathcal{T}_G(H)=N_G(H),
  && \mathcal{L}_G(H)=O^pC_G(H)\backslash N_G(H),\\
  \mathcal{F}_G(H)=C_G(H)\backslash
  N_G(H),&&\mathcal{O}_G(H)=N_G(H)/H,
  &&\catt FG{}(H)=C_G(H)\backslash N_G(H)/H,
\end{array}
\]
The six categories  are related by the commutative diagram
\begin{equation*}
  \xymatrix@R=15pt{
    {\cat SG{}} \ar@{^(->}[r] &
    {\cat TG{}} \ar@{->>}[r] \ar@{->>}[d] &
    {\cat LG{}} \ar@{->>}[r] & 
    {\cat FG{}} \ar@{->>}[d] \\
    & {\cat OG{}} \ar@{->>}[rr]
    && {\widetilde{\mathcal{F}}_G}   } 
\end{equation*}
of one faithful and five full functors.

$\cat SG{}$ contains the identity subgroup as an initial object, so it is contractible.  More topologically interesting is the \emph{Brown poset} $\cat SG*$ of all nonidentity subgroups of $G$, which has long been an object of interest as a sort of geometry for the finite group $G$ (cf. \cite{brown75,quillen78}).  More generally, decorating one of our $p$-subgroup categories with an asterisk will denote the full subcategory of nonidentity subgroups:  $\cat TG*,\cat FG*$, etc.

Sylow's Theorem implies that each of our $p$-subgroup categories (other $\cat SG{}$) is equivalent to its full subcategory with objects  the subgroups of a fixed Sylow $p$-subgroup $P\in\mathrm{Syl}_p(G)$.  We will prefer to work with these ``pointed'' versions, especially as this convention allows us to work with \emph{abstract} $p$-subgroup categories that make no reference to an ambient group $G$:

Fix a finite nonidentity $p$-group $P$. An \emph{Frobenius $P$-category},  or
\emph{(abstract) saturated fusion system over $P$}, is a category whose objects are the
subgroups of $P$ and whose \m s satisfy a set of axioms
\cite{puig09,blo03} that distill the properties of the Frobenius
categories $\cat FG{}$ coming from a group $G$. There are examples of
abstract Frobenius $P$-categories $\cat F{}{}$ that are exotic in the
sense that there is no finite group $G$ with $\cat FG{}$ equal to
$\cat F{}{}$. 

The \emph{exterior quotient}, or  \emph{orbit category},  of
$\cat F{}{}$ is the category $\catt F{}{}$ whose objects are the subgroups of $P$
and with \m\ sets
\begin{equation*}
  \catt F{}{}(H,K) = \cat F{}{}(H,K)/\cat FK{}(K)
\end{equation*}
are $\cat F{}{}$-\m s modulo inner auto\m s of the codomain.
Composition in $\cat F{}{}$ induces composition in its quotient
category $\catt F{}{}$.

Finally,  we recall the fundamental notions of $G$- and $\cat
F{}{}$-selfcentralizing and $G$- and $\cat F{}{}$-radical subgroups. As usual, $O_pK$ is the
largest normal $p$-subgroup of the finite group $K$ \cite[Chp
6.3]{gorenstein68}.

\begin{defn}\cite[4.8.1]{puig09} \cite[Definition A.3]{blo03} 
  \cite[Proposition 4]{bouc84a}\label{defn:Grad}
  The $p$-subgroup $H$ of $G$ is
  \begin{itemize}
  \item \emph{$p$-selfcentralizing (in $G$)} if the center $Z(H)$ of $H$ is a \syl p\ of
    the centralizer $C_G(H)$ of $H$;
  %$C_H(H) \to C_G(H)$ is a $p$-Sylow inclusion.
   \item \emph{$G$-radical} if $O_p\cat OG{}(H)=1$, or, equivalently,
     $H=O_pN_G(H)$.  
  \end{itemize}
\end{defn}

  \begin{defn}\cite[4.8]{puig09} \cite[Definition A.9]{blo03} 
  \label{defn:radical}\label{defn:Frad}
    An object $H$ of  $\cat F{}{}$ is
    \begin{itemize}
    \item \emph{$\cat F{}{}$-selfcentralizing}
    if $C_P(H^\varphi) \leq H^\varphi$ for every $\cat F{}{}$-\m\
    $\varphi \in \cat F{}{}(H,P)$ with domain $H$
  \item  \emph{$\cat F{}{}$-radical} if $O_p \catt F{}{}(H) = 1$
    \end{itemize}
  \end{defn}

  Every $p$-selfcentralizing subgroup of $G$ is nontrivial, as is every
  $\cat F{}{}$-selfcentralizing subgroup of $P$.

  Let $P\in$ be a \syl p of $G$ and $\cat F{}{}=\cat FG{}$ the induced Frobenius $P$-category. For every 
  $H\leq P$,
  \begin{gather*}
     \text{$H$ is $\cat F{}{}$-selfcentralizing} \iff
     \text{$H$ is $p$-selfcentralizing in $G$} \\
     \text{$H$ is $\cat F{}{}$-selfcentralizing and $\cat
       F{}{}$-radical} \Longrightarrow
    \text{$H$ is $G$-radical}
  \end{gather*}
and the second implication can not be reversed.

According to Quillen \cite[Proposition 2.4]{quillen78},
we have 
\begin{equation}\label{eq:SK*contract}
  \text{$\cat SK*$  is noncontractible}
    \Longrightarrow O_pK=1
\end{equation}
for any finite group $K$.
 In the present context, this means that
\begin{align}
  &\text{$\cat S{\cat OG{}(H)}*$ is noncontractible}
    \Longrightarrow 
    \text{$H$ is $G$-radical} \label{eq:Grad} \\
  &  \text{$\cat S{\catt F{}{}(H)}{*}$ is noncontractible} 
   \Longrightarrow
   \text{$H$ is $\cat F{}{}$-radical} \label{eq:Frad}
\end{align}
Properties \eqref{eq:Grad} and \eqref{eq:Frad} will be very important
in the proof of our homotopy equivalences. (Quillen conjectures in
\cite[Conjecture~2.9]{quillen78} that the reverse implication of
\eqref{eq:SK*contract} is true.  If Quillen's conjecture holds,
then the reverse implications of \eqref{eq:Grad} and
\eqref{eq:Frad} are true as well.)

\section{Weightings and coweightings for $\mathrm{EI}$-categories}
\label{sec:weundercat}

Let $\cat C{}{}$ be a finite category and $[\cat C{}{}]$ the set of
iso\m\ classes of its objects. In this section we show that we can use
coslice categories to define \we s in the sense of Leinster
\cite{leinster08}.  In particular, we will relate the Euler characteristic of $\cat C{}{}$ to is coslice categories.

\begin{defn}\cite[Definitions~1.10, 2.2]{leinster08}\label{defn:wecowe}
  A \emph{\we}\ for $\cat C{}{}$ is a function $k^\bullet_{\cat C{}{}} \colon
  \Ob{\cat C{}{}} \to \Q$  so that
\begin{equation*}
  \forall a \in  \Ob{\cat C{}{}} \colon \sum_{b \in \Ob{\cat C{}{}}} 
  |\cat C{}{}(a,b)|k^b_{\cat C{}{}} = 1,
\end{equation*}
and a \emph{co\we}\ for $\cat C{}{}$ is a function $k_\bullet^{\cat C{}{}}
\colon \Ob{\cat C{}{}} \to \Q$ so that
  \begin{equation*}
  \forall b \in  \Ob{\cat C{}{}} \colon 
  \sum_{a \in \Ob{\cat C{}{}}} 
  k_a^{\cat C{}{}} |\cat C{}{}(a,b)| = 1.
  \end{equation*}
If $\cat C{}{}$ has both a \we\ and a co\we , then 
\begin{equation*}
  \sum_{a \in \Ob{\cat C{}{}}} k^a_{\cat C{}{}}\  = \sum_{b \in \Ob{\cat C{}{}}} k_b^{\cat C{}{}}\ =:\chi(\cat C{}{})
\end{equation*}
 is the \emph{\Euc}\ of $\cat C{}{}$. The {\em reduced\/}  \Euc\ of $\cat
C{}{}$ is $\rchi(\cat C{}{}) = \chi(\cat C{}{}) - 1$.
\end{defn}

  A general finite category may not admit a \we\ or co\we,  or it may
  have several.  If at least one of each exists,   the \Euc\ is
  independent of the choice of \we\ or co\we .  Moreover, if there are several (co)\we s, we will normalize  by singling out the unique (co)\we\ that is constant on each isomorphism class of our category.
  
  We will be interested in piecing together global (co)\we s from local data, to be described in terms of the (co)slice construction.

  \begin{defn}[Coslice and slice categories]\label{defn:C//c}
     Let $\cat C{}{}$ be a category with objects $x$ and $y$, and $\cat A{}{}$ a full
    subcategory of $\cat C{}{}$.

\begin{itemize}

\item $x/\cat A{}{}$ is the category of
$\cat C{}{}$-\m s from $x$ to an object of $\cat A{}{}$ 
(the \emph{coslice} of $\cat A{}{}$ under $x$)
%\raisebox{4.0cm}[0pt][0pt]{\hspace{1cm}$\xymatrix{
%  & x  \ar@{}[d]|(.6)*{x//\cat A{}{}}
%       \ar[dl] \ar[dr] \\
%  a_1 \ar@{.>}[rr] && a_2}$} 

\item $\cat A{}{}/y$ is the category of $\cat C{}{}$-\m s from an
  object of $\cat A{}{}$ to $y$
  (the \emph{slice} of $\cat A{}{}$ over $y$)
%\raisebox{1.0cm}[0pt][0pt]{\hspace{1.5cm}
%$\xymatrix{
%  a_1 \ar[dr] \ar@{.>}[rr] && a_2 \ar[dl] \\
%  & y  \ar@{}[u]|(.6)*{\cat A{}{}//y}}$}

  \item $x//\cat A{}{}$ is the full subcategory of $x/\cat A{}{}$
    with objects all noniso\m s from $x$ to an object of $\cat
    A{}{}$ (the \emph{proper coslice} of $\cat A{}{}$ under $x$)

  \item $\cat A{}{}//y$ is the full subcategory of $\cat A{}{}/y$
    with objects all noniso\m s from an object of $\cat A{}{}$ to $y$
    (the \emph{proper slice} of $\cat A{}{}$ over $y$)

\end{itemize}

Thus an object of $x/\cat A{}{}$ is a $\cat C{}{}$-morphism $\varphi_1\in\cat C{}{}(x,a_1)$, and a $(x/\cat C{}{})$-morphism from $\varphi_1$ to $\varphi_2\in\cat C{}{}(x,a_2)$ is any $\eta\in\cat C{}{}(a_1,a_2)$ that makes the following triangle  commute  in $\cat C{}{}$:  $\xymatrix{&x\ar[dl]_{\varphi_1}\ar[dr]^{\varphi_2}\ar@{}[d]|-\circlearrowleft
&\\
a_1\ar@{.>}[rr]_\eta&&a_2}$

The other (co)slice categories are defined similarly.

\end{defn}

The (co)slice constructions define functors $\bullet/\cat A{}{}:\cat C{}{\mathrm{op}}\to\mathrm{CAT}$ and $\cat A{}{}/\bullet:\cat C{}{}\to\mathrm{CAT}$ via (pre)composition of morphisms.  The proper (co)slice constructions do not in general piece together to form a functor, for the simple reason that the composite of two nonisomorphisms can be an isomorphism.  However, if $\cat C{}{}$ is an EI-category, this obstruction vanishes and we also have functors $ \bullet // \cat A{}{} \colon \cat C{}{\mathrm{op}} \to \mathrm{CAT}$ and $\cat A{}{}//\bullet \colon
  \cat C{}{} \to \mathrm{CAT}$.  Again, all $p$-subgroup categories considered in this paper are EI-categories.

\begin{defn}
 If $\cat C{}{}$ is an EI-category, we write
  \begin{alignat*}{3}
    &\mathrm{supp}( \bullet / \cat A{}{}) = \{ x \in \Ob{\cat C{}{}}
    \mid \text{$x / \cat A{}{}$ is noncontractible} \} & &\qquad &
    &\mathrm{supp}( \cat A{}{} / \bullet) = \{ y \in \Ob{\cat C{}{}}
    \mid \text{$ \cat A{}{} / y$ is noncontractible} \} \\
    &\mathrm{supp}( \bullet // \cat A{}{}) = \{ x \in \Ob{\cat C{}{}}
    \mid \text{$x // \cat A{}{}$ is noncontractible} \}& &\qquad &
    &\mathrm{supp}( \cat A{}{} // \bullet) = \{ y \in \Ob{\cat C{}{}}
    \mid \text{$ \cat A{}{} // y$ is noncontractible} \}
  \end{alignat*}
  for the \emph{supports} of the coslice functors $\bullet / \cat A{}{},
  \bullet // \cat A{}{} \colon \cat C{}{\mathrm{op}} \to \mathrm{CAT}$
  and slice functors $\cat A{}{}/\bullet, \cat A{}{}//\bullet \colon
  \cat C{}{} \to \mathrm{CAT}$.
\end{defn}

The notation  $\cat A{}{}//\bullet$ and  $\bullet//\cat A{}{}$ is
taken from \cite[p $269$]{jackowski_slominska2001}.
% Same notation as in \cite{jackowski_slominska2001} for $\cat
% A{}{}//\bullet$.

\begin{lemma}\label{lemma:weundercat}
Let $\cat C{}{}$ be any finite category  admitting a \we\
$k^\bullet_{\cat C{}{}} \colon \Ob{\cat C{}{}} \to \Q$. Let $a$ be any
object of $\cat C{}{}$.  The function
\begin{equation*}
  k^\bullet_{a/\cat C{}{}} =
  k_{\cat C{}{}}^{\mathrm{cod}(\bullet)}
  \colon \Ob{a/\cat C{}{}} \to \Q, \qquad
  k_{a/\cat C{}{}}^{a \stackrel{\varphi}\to b} = k^b_{\cat C{}{}}, 
\end{equation*}
is a \we\ for the coslice $a/\cat C{}{}$ of $\cat C{}{}$
under $a$.
\end{lemma}
\begin{proof}
  The set of objects of $a/\cat C{}{}$, which is the set   of $\cat
  C{}{}$-\m s with domain $a$, is partitioned 
  \begin{equation}\label{eq:Oba/C}
    \Ob{a/\cat C{}{}} = \coprod_{b \in \Ob{\cat C{}{}}} \cat C{}{}(a,b)
  \end{equation}
  according to codomains. Also, for any $\cat C{}{}$-\m\ $\varphi \in \cat
  C{}{}(a,b)$ with codomain $b$ and any $\cat C{}{}$-object $c$, the
  set of $\cat C{}{}$-\m s from $b$ to $c$ is partitioned
 \begin{equation}\label{eq:Crs}
   \cat C{}{}(b,c) = \coprod_{\psi \in \cat C{}{}(a,c)} (a/\cat C{}{})(\varphi,\psi)
 \end{equation}
  into $a/\cat C{}{}$-\m\ sets with domain $\varphi$.
  The computation
  \begin{equation*}
    \sum_{\psi \in \Ob{a/\cat C{}{}}}
      |a/\cat C{}{}(\varphi,\psi)|  k_{\cat C{}{}}^{\mathrm{cod}(\psi)} 
      \stackrel{\eqref{eq:Oba/C}}{=}
      \sum_{b \in \Ob{\cat C{}{}}} \sum_{\psi \in \cat C{}{}(a,b)}
      |a/\cat C{}{}(\varphi,\psi)|  k_{\cat C{}{}}^{\mathrm{cod}(\psi)} 
      \stackrel{\eqref{eq:Crs}}{=}
       \sum_{b \in \Ob{\cat C{}{}}} |\cat C{}{}(\mathrm{cod}(\varphi),b)|  k_{\cat
         C{}{}}^{b} = 1
  \end{equation*}
  shows that the function $k_{\cat C{}{}}^{\mathrm{cod}(\bullet)}$ is
  a \we\ on $a/\cat C{}{}$.
\end{proof}

Here is a variation \cite[2.4]{jmm_mwj:2010} of
Definition~\ref{defn:wecowe}: For any object $a$ of $\cat C{}{}$,
write $[a] \in [\cat C{}{}]$ for the set of objects isomorphic to $a$.  Viewing $[\mathcal{C}]$ as a skeletal category by picking a representative from each $\cat C{}{}$-isomorphism class, 
the matrix for $[\cat C{}{}]$ is the square $(|[\cat C{}{}]| \times
|[\cat C{}{}]|)$-matrix with entries  $[\cat C{}{}]([a],[b]) =
|\cat C{}{}(a,b)|$, $[a], [b] \in [\cat C{}{}]$. A  \emph{\we} for $[\cat
C{}{}]$ is a function $k_{[\cat C{}{}]}^\bullet \colon [\cat C{}{}]
\to \Q$ so that
\begin{equation*}
  \forall [a] \in  [\cat C{}{}] \colon \sum_{[b] \in [\cat C{}{}]} 
  \big|[\cat C{}{}]([a],[b])\big|\cdot k^{[b]}_{[\cat C{}{}]} = 1,
\end{equation*}
and a \emph{co\we} for  $[\cat C{}{}]$ is a function $k^{[\cat
  C{}{}]}_\bullet \colon [\cat C{}{}] \to \Q$ so that 
  \begin{equation*}
  \forall [b] \in  [\cat C{}{}] \colon 
  \sum_{[a] \in [\cat C{}{}]} 
  k_{[a]}^{[\cat C{}{}]}\cdot\big| [\cat C{}{}]([a],[b])\big| = 1.
  \end{equation*}
  The category $\cat C{}{}$ has a \we\ $k^\bullet_{\cat C{}{}}$ if and
  only its set of iso\m\ classes of objects $[\cat C{}{}]$ has a
  \we\ $k^\bullet_{[\cat C{}{}]}$. We can construct one from the other:
\begin{equation*}
    k_{[\cat C{}{}]}^{[b]} = \sum_{y \in [b]}  k_{\cat C{}{}}^y, \qquad
    k_{\cat C{}{}}^b = \frac{1}{\big|[b]\big|} \cdot k_{[\cat C{}{}]}^{[b]}, \qquad
    b \in \Ob{\cat C{}{}}, [b] \in [\cat C{}{}].
\end{equation*}
Similarly, $\cat C{}{}$ has a co\we\ if and only if $[\cat C{}{}]$ has
a co\we . 
Note that the (co)weightings on $\cat C{}{}$ that arise in this manner are necessarily constant on $\cat C{}{}$-isomorphism classes of objects.

If $\cat C{}{}$ is an EI-category, its objects can be arranged in such an order so that the matrix $[\cat C{}{}]$ is upper triangular.   It follows that any finite
$\mathrm{EI}$-category  has a {\em unique\/} \we\ and a
{\em unique\/} co\we\ that are constant on iso\m\ classes of objects
\cite[Lemma~1.3, Theorem~1.4, Lemma~1.12]{leinster08}.

A full subcategory $\mathcal{I}$ of a category $\mathcal{C}$ is a {\em
  left ideal\/} if any $\mathcal{C}$-\m\ whose domain is an object of
$\mathcal{I}$ is an $\mathcal{I}$-\m. For instance, if $\mathcal{C}$
is an $\mathrm{EI}$-category and $a$ an object of $\mathcal{C}$ then
$a//\cat C{}{}$ is a left ideal in $a/\cat C{}{}$ by \cite[Lemma
1.3]{leinster08}.

\begin{thm}\label{lemma:underwt}
  Let $\cat C{}{}$ be a finite $\mathrm{EI}$-category, and let
  $k^\bullet_{\cat C{}{}}$ and $k_\bullet^{\cat C{}{}}$ be the \we\ and
  the co\we\ on $\cat C{}{}$ that are constant on iso\m\ classes of
  objects of $\cat C{}{}$. 
  Then
  \begin{equation*}
     k^a_{\cat C{}{}} = 
    \frac{-\rchi(a//\cat C{}{})}{ |[a]| |\cat C{}{}(a)|},
    \qquad
      k_b^{\cat C{}{}} = 
    \frac{-\rchi(\cat C{}{}//b)}{ |[b]| |\cat C{}{}(b)|}, 
    \qquad a,b \in \Ob{\cat C{}{}},
  \end{equation*}
  and the \Euc\ of $\cat C{}{}$ is
   \begin{equation*}
    \sum_{[a] \in [\cat C{}{}]} 
    \frac{-\rchi(a//\cat C{}{})}{ |\cat C{}{}(a)|} =
     \chi(\cat C{}{}) =
     \sum_{[b] \in [\cat C{}{}]} 
    \frac{-\rchi(\cat C{}{}//b)}{ |\cat C{}{}(b)|}
  \end{equation*}
  where the sums run over the set $[\cat C{}{}]$ of iso\m\ classes of
  objects of $\cat C{}{}$.
\end{thm}
\begin{proof}
  We shall only prove the statement about the \we\ since the statement
  about the co\we\ is entirely dual.   $\cat C {}{}$ is a finite
  $\mathrm{EI}$-category, so it is easy to see that the coslice categories $a/\cat C{}{}$ and
  $a//\cat C{}{}$ are also finite $\mathrm{EI}$-categories. Thus they
  admit \we s and co\we s, and have well-defined Euler characteristics. Since $a//\cat C{}{}$ is a left ideal in
  $a/\cat C{}{}$, the \we\ for $a/\cat C{}{}$ from
  Lemma~\ref{lemma:weundercat} restricts to a \we\ for $a//\cat C{}{}$
  \cite[Remark 2.6]{jmm_mwj:2010}. The category $a/\cat C{}{}$ has an
  initial object, so it is contractible and has Euler characteristic
  1.  Therefore
  \begin{equation*}
    1 = 
    \sum_{\varphi \in \Ob{a/\cat C{}{}}}  
    k^{\mathrm{cod}(\varphi)}_{\cat C{}{}} =
    |[a]| |\cat C{}{}(a)|  k^{a}_{\cat C{}{}} +
      \sum_{\varphi \in \Ob{a//\cat C{}{}}}  
   k^{\mathrm{cod}(\varphi)}_{\cat C{}{}}=  
    |[a]| |\cat C{}{}(a)|  k^{a}_{\cat C{}{}} + \chi(a // \cat C{}{})
  \end{equation*}
  because the \we\ $k^\bullet_{\cat C{}{}}$ is assumed to be constant
  on the iso\m\ class $[a]$ of $a$.
\end{proof}

The rational functions
 \begin{equation*}
    k^{[a]}_{[\cat C{}{}]} = 
    \frac{-\rchi(a//\cat C{}{})}{ |\cat C{}{}(a)|},
    \qquad
     k_{[b]}^{[\cat C{}{}]} = 
    \frac{-\rchi(\cat C{}{}//b)}{ |\cat C{}{}(b)|}, 
    \qquad [a], [b] \in [\Ob{\cat C{}{}}],
  \end{equation*}
  are the \we\ and the co\we\ for $[\cat C{}{}]$, respectively.

%   \begin{align*}
%     \sum -\rchi(\cat S{\cat OG{}(H)}*) &= \sum -\mu(K) \qquad (\cat SG*) \\
%     \sum \frac{-\rchi(\cat S{\cat OG{}(H)}*)}{|N_G(H)|} &= 
%     \sum \frac{-\mu(K)}{|N_G(K)|} \qquad (\cat TG*) \\
%     \sum \frac{-\rchi(\cat S{\cat OG{}(H)}*)}{|\cat OG{}(H)|} &=
%     \frac{1+(p-1)\sum_C |C|}{p|G|} \qquad (\cat OG{})
%   \end{align*}

In the case $\cat S{}{}$ is a poset, we sometimes write 
${}_{a \leq}\cat S{}{}$, 
${}_{a<}\cat S{}{}$,
${\cat S{}{}}_{ \leq b}$,
${\cat S{}{}}_{<b}$ for
$a/\cat S{}{}$,
$a//\cat S{}{}$,
$\cat S{}{}/b$,
$\cat S{}{}//b$, respectively. Using this notation, the last part of
Theorem~\ref{lemma:underwt} takes the following form. 

\begin{cor}\label{cor:underwt}
  The \Euc\ of a finite poset  $\cat S{}{}$ is the sum
  \begin{equation*}
     \sum_{a \in \Ob{\cat S{}{}}} -\rchi({}_{a<}\cat S{}{}) = \chi(\cat
     S{}{}) =
     \sum_{b \in \Ob{\cat S{}{}}} -\rchi({\cat S{}{}}_{<b})
  \end{equation*}
  of the negatives of the local reduced \Euc s.
\end{cor}

This reproduces a well-known result from the combinatorial theory of
posets. 

% \begin{exmp}
%   Let $\cat S{}{}$ be a poset with the four elements $\{1,2,3,4\}$
%   partially ordered by the relations $1<3$, $2 < 3$, $2 < 4$, and $3 <
%   4$. The $\zeta$ and $\mu$-matrices are
%   \begin{equation*}
%     \zeta =
%     \begin{pmatrix}
%       1 & 0 & 1 & 0 \\ 0 & 1 & 1 & 1 \\ 0& 0 & 1 & 0 \\0 & 0 & 0 & 1
%     \end{pmatrix}, \qquad
%     \mu =
%     \begin{pmatrix}
%       1 & 0 & -1 & 0 \\ 0 & 1 & -1 & -1 \\ 0& 0 & 1 & 0 \\0 & 0 & 0 & 1
%     \end{pmatrix}
%   \end{equation*}
%   The \we\ and the co\we\ 
%   \begin{equation*}
%     k^\bullet = (0,-1,1,1), \qquad k_\bullet = (1,1,-1,0)
%   \end{equation*}
%   record the row sums and the column sums in $\mu$.  The strict
%   (co)slice categories $\bullet // \cat S{}{} ={}_{\bullet <}{\cat
%     S{}{}}$ and $\cat S{}{} // \bullet = {\cat S{}{}}_{< \bullet}$ are
%   discrete posets on zero, one, or two objects so their \Euc s are
%   \begin{equation*}
%     \chi(\bullet // \cat S{}{}) = (1,2,0,0), \qquad
%     \chi(\cat S{}{} // \bullet) = (0,0,2,1)
%   \end{equation*}
%   and their reduced \Euc s
%   \begin{equation*}
%      \rchi(\bullet // \cat S{}{}) = (0,1,-1,-1), \qquad
%      \rchi(\cat S{}{} // \bullet) = (-1,-1,1,0)
%   \end{equation*}
%   confirming that $k^\bullet = - \rchi(\bullet // \cat S{}{})$ and
%   $k_\bullet = - \rchi( \cat S{}{} // \bullet)$.
% \end{exmp}

  \section{Homotopy equivalences between categories}
  \label{sec:bouc}

  The famous Quillen's Theorem A provides a sufficiency criterion for a
  functor between two categories to be a homotopy equivalence.  We
  quote this Theorem here not in its full generality, but only for
  the special case that is of interest to us.

  \begin{thm}[Quillen's Theorem~A for inclusions of categories] \label{thm:QuillenA}
    \cite[Theorem~A]{quillen73}
  %   (\href{http://math.mit.edu/~mbehrens/TAGS/Isaacson_exer.pdf}{Isaacson
%       notes}) 
    Let $\cat C{}{}$ be a category and $\cat A{}{}$ a full
    subcategory. The inclusion $\cat A{}{} \hookrightarrow \cat C{}{}$
    is a homotopy equivalence if either $\mathrm{supp}(\bullet/\cat
    A{}{})$ or $\mathrm{supp}(\cat A{}{}/\bullet)$ is empty.
%%%%%%%%%%%%%%%%%%%%%%%%%%%%%
%    Suppose that at least one of the following two
%   conditions holds
%     \begin{enumerate}[(a)]
%     \item $x/\cat A{}{}$ is contractible for all objects $x$ of $\cat
%       C{}{}$ (outside the subcategory $\cat A{}{}$)
%     \item $\cat A{}{}/y$ is contractible for all objects $y$ of $\cat
%       C{}{}$ (outside the subcategory $\cat A{}{}$)
%     \end{enumerate}
%  Then the inclusion $\cat A{}{} \hookrightarrow \cat C{}{}$ is a
%     homotopy equivalence.
\end{thm}

%   \begin{thm}[Quillen's Theorem~A for
%     posets]\label{lemma:quillenAposet} \cite[Theorem A]{quillen73}
%     Let $\cat S{}{}$ be a finite poset and $\cat A{}{}$ a subposet.
%     The inclusion $\cat A{}{} \hookrightarrow \cat S{}{}$ is a
%     homotopy equivalence if either $\mathrm{supp}(\bullet/\cat A{}{})$
%     or $\mathrm{supp}(\cat A{}{}/\bullet)$ is empty.
%   \end{thm}

We also quote a perhaps less well-known result of Bouc providing a
sufficient condition for an inclusion of posets to be a homotopy
equivalence. 

  \begin{thm}[Bouc's theorem for posets]\label{lemma:bouc} \cite{bouc84a}
    Let $\cat S{}{}$ be a finite poset and $\cat A{}{}$ a subposet.
    The inclusion $\cat A{}{} \hookrightarrow \cat S{}{}$ is a
    homotopy equivalence if either $\mathrm{supp}(\bullet//\cat
    S{}{})$ or $\mathrm{supp}(\cat S{}{}//\bullet)$ is contained in
    $\Ob{\cat A{}{}}$.
  \end{thm}

%The philosophy is that the objects outside the subposet $\cat A{}{}$
%do not contribute anything to $\cat S{}{}$ anyway. 

  In this section we generalize Bouc's theorem for poset inclusions to
  finite $\mathrm{EI}$-category inclusions.  This  boils down to a repurposing of Quillen's Theorem~A in terms of our notion of support, and should be thought of as a statement about what sort of objects control the homotopy type of a finite EI-category.

%% Bouc's result for posets extends to finite $\mathrm{EI}$-categories.

 %  \begin{lemma}(\href{http://dl.dropbox.com/u/27907946/EI%20categories.2.pdf}
% {Bouc theory for $\mathrm{EI}$-categories})
%     \label{lemma:boucEI} 
%     Let $\cat C{}{}$ be a finite $\mathrm{EI}$-category and $\cat
%     A{}{}$ a full subcategory that is closed under iso\m s.  Suppose
%     that at least one of the following two conditions holds
%     \begin{enumerate}[(a)]
    
%      \item $x//\cat C{}{}$ is contractible for all objects $x$ outside
%       the subcategory $\cat A{}{}$ \label{lemma:boucEIb}

%     \item $\cat C{}{}//y$ is contractible for all objects $y$ outside the
%       subcategory $\cat A{}{}$ \label{lemma:boucEIa}

%     \end{enumerate}
%     Then the inclusion $\cat A{}{} \hookrightarrow \cat C{}{}$ is a
%     homotopy equivalence.
%   \end{lemma}

 %  \begin{lemma}(\href{http://dl.dropbox.com/u/27907946/EI%20categories.2.pdf}
% {Bouc theory for $\mathrm{EI}$-categories})
%     \label{lemma:boucEI2} 
%     Let $\cat C{}{}$ be a finite $\mathrm{EI}$-category and $\cat
%     A{}{}$ a full subcategory that is closed under iso\m s.  Suppose
%     that at least one of the following two conditions holds
%     \begin{enumerate}[(a)]
    
%     \item The full subcategory $\cat A{}{}$ contains all objects $x$
%       of $\cat C{}{}$ for which $x//\cat C{}{}$ is noncontractible
%       \label{lemma:boucEI2b}

%     \item The full subcategory $\cat A{}{}$ contains all objects $x$
%       of $\cat C{}{}$ for which $x//\cat C{}{}$ is noncontractible
%       \label{lemma:boucEI2a}

%     \end{enumerate}
%     Then the inclusion $\cat A{}{} \hookrightarrow \cat C{}{}$ is a
%     homotopy equivalence.
%   \end{lemma}

  \begin{thm}[Bouc's theorem for finite $\mathrm{EI}$-categories]
    \label{lemma:boucEI2} 
    Let $\cat C{}{}$ be a finite $\mathrm{EI}$-category and $\cat
    A{}{}$ a full subcategory that is closed under iso\m s.  The
    inclusion of $\cat A{}{} \hookrightarrow \cat C{}{}$ is a homotopy
    equivalence if either $\mathrm{supp}(\bullet // \cat C{}{})$ or
    $\mathrm{supp}(\cat C{}{} // \bullet)$ is contained in $\Ob{\cat
      A{}{}}$.
  \end{thm}
  \begin{proof}
    Assume that $\Ob{\cat A{}{}}$ contains the support
    $\mathrm{supp}(\bullet//\cat C{}{})$ of the functor $\bullet//\cat
    C{}{}$. The claim is that the inclusion functor $\func{\iota}{\cat
      A{}{}}{\cat C{}{}}$ is a homotopy equivalence. It suffices to
    show that the coslice $x/\cat A{}{}$ of $\cat A{}{}$ is
    contractible for every object $x$ of $\cat C{}{}$
    (Theorem~\ref{thm:QuillenA}).

    For any object $x$ of $\cat C{}{}$ define the {\em height\/} of
    $x$, $\mathrm{ht}(x)$, to be the maximal length of any path
    \begin{equation*}
      x_0 \to x_1 \to \cdots \to x_h = x
    \end{equation*}
    of noniso\m s in $\cat C{}{}$ terminating at $x$. The height of
    $x$ is finite since there are no circuits in paths of noniso\m s
    \cite[Lemma 1.3]{leinster08}. If there is a noniso\m\ from $x_0$
    to $x_1$, then $\mathrm{ht}(x_0)<\mathrm{ht}(x_1)$. Define
    $\mathrm{ht}(\cat C{}{})$ to be the maximal height of any object
    of $\cat C{}{}$.

    Suppose that $x$ is an object of $\cat C{}{}$ of maximal height,
    $\mathrm{ht}(\cat C{}{})$.  Then $x//\cat C{}{}$ is the empty
    category because there is no noniso\m\ from $x$ to any object of
    $\cat C{}{}$. The empty category is not contractible, so $x \in
    \mathrm{supp}(\bullet//\cat C{}{}) \subset \Ob{\cat A{}{}}$ is an
    object of $\cat A{}{}$. Then $x/\cat A{}{}$ is contractible with
    the identity of $x$ as an initial object.

    Let now $x$ be any object of $\cat C{}{}$ such that the coslice
    $y/\cat A{}{}$ of $\cat A{}{}$ is contractible for all objects $y$
    of height greater than $\mathrm{ht}(x)$. Then the functor
   \begin{equation*}
     x//\iota \colon x//\cat A{}{} \to x//\cat C{}{}
   \end{equation*}
   is a homotopy equivalence by Theorem~\ref{thm:QuillenA} because the
   category
   \begin{equation*}
     (x \to y)/(x//\iota) = y/\cat A{}{}
   \end{equation*}
   is contractible for every object $x \to y$ of $x//\cat C{}{}$. In
   the case $x$ is an object of $\cat A{}{}$, $x/\cat A{}{}$ is
   contractible as before. In the case $x$ is not an object of $\cat
   A{}{}$, $x/\cat A{}{} = x//\cat A{}{}$ because there can be no
   iso\m\ from $x$ to an object of $\cat A{}{}$ as $\cat A{}{}$ is
   closed under iso\m s. We now have
   \begin{equation*}
     x/\cat A{}{} = x//\cat A{}{} \simeq x//\cat C{}{}
   \end{equation*}
   and $x//\cat C{}{}$ is contractible since $x \not\in
   \mathrm{supp}(\bullet//\cat C{}{})$. Thus $x / \cat A{}{}$ is also
   contractible.

   By finite downward induction on $\mathrm{ht}(x)$ we see that $x /
   \cat A{}{}$ is contractible for all objects $x$ of $\cat C{}{}$.
  \end{proof}

\begin{rmk}\label{rmk:BoucInterpretation}
Theorem~\ref{lemma:boucEI2} is the main technical tool of this paper, but it should also be thought of as a descriptive statement about control of  homotopy type of  finite EI-categories.  We will  use the following reinterpretation:  If $\cat C{}{}$ is a finite EI-category, then either (a) those objects whose proper slice categories are contractible do not, as a whole, contribute to the overall homotopy type of $\cat C{}{}$, or (b) the same holds for those objects with contractible proper coslice categories.  (Note that the union of these two classes cannot be discarded without affecting the homotopy type, as the example of the poset $x<y$ shows.)  In our search for homotopy equivalences between $p$-subgroup categories, we will therefore concentrate on identifying those objects with contractible proper (co)slices.  Since the reduced Euler characteristic of a contractible category is $0$, we will use the combinatorics developed in the previous section to direct our search in what follows.
\end{rmk}

  \section{Subgroup categories for $p$-groups}
  \label{sec:pgroups}

  In this section we collect several  technical examples that will allow us to apply Theorem~\ref{lemma:boucEI2} more generally.

For any small category $\cat C{}{}$ and any set $D \subset \Ob{\cat
  C{}{}}$ of objects of $\cat C{}{}$, we let $\cat C{}D$ denote the
full subcategory of $\cat C{}{}$ generated by the objects in the set
$D$. For instance, if $H \lneq K$ are $p$-subgroups of $G$, then $\cat
FG{[H,K)}$ denotes the full subcategory of $\cat FG{}$ with objects
the set of all subgroups $L$ of $G$ for which $H \leq L \lneq K$.

In the following lemma we consider
 \begin{description}
 \item[$\cat SP{(1,P)}$] the poset of nonidentity  proper
   subgroups of $P$

% \item $\frac{(1,P)}{P}$ or $\frac{\cat SP{(1,P)}}{P}$ the
%   poset of $P$-conjugacy classes of nonidentity and proper subgroups
%   of $P$

\item[$\cat OP{[1,P)}$] the full subcategory of $\cat OP{}$
  with objects all proper subgroups of $P$

\item[$\catt FP{(1,P)}$] the full subcategory of $\catt FP{}$
  with objects all nonidentity proper subgroups of $P$
 \end{description}
 for $P$ a nonidentity $p$-group. We write $\mu$ for the M\"obius
 function of the poset $\cat SP{}$ \cite[\S 3.7]{stanley97}, and we
 abbreviate $\mu(1,K)$ to $\mu(K)$ for any subgroup $K$ of $P$.

\begin{lemma}\label{lemma:(1,P)contract}
  Let $P$ be a nonidentity $p$-subgroup. Then
  \begin{enumerate}[(a)]

  \item 
  \begin{itemize}
  \item $\rchi(\cat SP{(1,P)}) = \mu(P)$
  \item $\rchi(\catt FP{(1,P)}) =
   \rchi(\cat FP{(1,P)})= \frac{\mu(P)}{|P : Z(P)|} $
    \item$\chi(\cat
    OP{[1,P)}) = \left\{\begin{array}{ll}p^{-1}&\textrm{P is cyclic}\\ 1&\textrm{else}\end{array}\right.$
    \end{itemize} \label{lemma:(1,P)contractB} 

 \item \label{lemma:(1,P)contractA}
    $\text{$\cat SP{(1,P)}$ is noncontractible} \iff
    \text{$P$ is elementary abelian}$.
 
% \item \label{lemma:(1,P)contractC}
%     $\text{$\frac{\cat SP{(1,P)}}{P}$ is not contractible} \iff
%     \text{$P$ is elementary abelian}$

  \item \label{lemma:(1,P)contractD} $\cat OP{[1,P)}$ is homotopy
    equivalent to $\cat OV{[1,V)}$, where $V= P/\Phi(P)$ is the
    Frattini quotient of $P$.

% \item  \label{lemma:(1,P)contractD}
%     $\text{$\cat OP{[1,P)}$ is noncontractible} \iff
%     \text{$P$ is cyclic}$ \mynote{ $\Leftarrow$ is TRUE, but
%       $\Rightarrow$ is FALSE?}

\item  \label{lemma:(1,P)contractE}
    $\text{$\catt FP{(1,P)}$ is noncontractible} \iff
    \text{$P$ is elementary abelian}$,

  \end{enumerate}

\end{lemma}
\begin{proof}

  \noindent \eqref{lemma:(1,P)contractB} It is well known that
  $\rchi(\cat SP{(1,P)})= \rchi(1,P) = \mu(P)$ \cite[3.8.5,
  3.8.6]{stanley97} \cite[2.3]{jmm_mwj:2010}.  The formulas for
  $\rchi(\catt FP{(1,P)})$ and $\chi(\cat OP{[1,P)})$ follow from
  \cite[Remark~2.6, Example~3.7, Theorem~7.7,
  Theorem~4.1]{jmm_mwj:2010}: If $k_\bullet^{\catt F{}{}}$ is a co\we\
  for $\catt FP{(1,P]}$ then
  \begin{equation*}
    1 = \chi(\catt FP{(1,P]}) =  \chi(\catt FP{(1,P)}) + k_P^{\catt
      F{}{}} =  \chi(\catt FP{(1,P)}) +   \frac{-\mu(P)}{|P : Z(P)|}
  \end{equation*}
  because $\chi(\catt FP{(1,P]})$ is contractible, with $P$ as terminal
  object, containing the right ideal $\chi(\catt
  FP{(1,P)})$. Similarly,  If $k_\bullet^{\cat O{}{}}$ is a co\we\
  for $\cat OP{}$ then
  \begin{equation*}
    1 = \chi(\cat OP{}) = 
    \chi(\cat OP{[1,P)}) + k_P^{\cat O{}{}} = \chi(\cat OP{[1,P)}) +
    \begin{cases}
      1-\frac{1}{p} & \text{$P$ is cyclic} \\
      0 & \text{$P$ is not cyclic}
    \end{cases}
  \end{equation*}
  and the expression for the \Euc\ of $\cat OP{[1,P)}$ follows.

  \noindent \eqref{lemma:(1,P)contractA} If $P$ is elementary abelian,
  the poset $\cat SP{(1,P)}$ is noncontractible because its reduced
  \Euc\ is nonzero according to \eqref{lemma:(1,P)contractB}.  If $P$
  is not elementary abelian, the Frattini subgroup $\Phi(P)$ is
  nontrivial \cite[Chp 5, Theorem 1.3]{gorenstein68}. There are
  adjoint functors
  \begin{equation*}
    \xymatrix@1{
         \cat SP{(1,P)}\ar@<.75ex>[r]^-L &
        \cat SP{[\Phi(P),P)}  \ar@<.75ex>[l]^-R \ar@<.75ex>[r]^-R &
         \cat SP{[1,P)}\ar@<.75ex>[l]^-L }     
  \end{equation*}
  where $QL=Q\Phi(P)$ and $QR = Q$ for $Q \leq P$. 
%%%%%
% \begin{equation*}
%       \cat SP{(1,P)} \to \cat SP{[\Phi(P),P)}, \qquad Q \to Q \cdot \Phi(P)
%     \end{equation*}
  Observe that $Q \lneqq P \Longrightarrow Q\Phi(P) \lneqq P$ because
  the Frattini subgroup is the group of nongenerators of $P$. The
  poset on the right, $\cat SP{[1,P)}$, is contractible with the
  trivial group as an initial object. The poset on the left, $ \cat
  SP{(1,P)}$, is therefore also contractible. Alternatively, the
  natural transformations $Q \leq Q\Phi(P) \geq \Phi(P)$, $1 \leq Q
  \lneqq P$, define a homotopy from the identity of $\cat SP{(1,P)}$
  to a constant map.
  %   By
%     Lemma~\ref{lemma:bouc}.\eqref{lemma:bouc1}, the inclusion $\cat
%     SG{\mathrm{eab}} \hookrightarrow \cat SG*$ is a homotopy
%     equivalence.

%     \noindent \eqref{lemma:(1,P)contractC} Note that $Q\Phi(P)$ is a
%     normal subgroup that only depends on the conjugacy class of $Q$.
%     By the same proof as above, $\frac{(1,P)}{P}$ is contractible if
%     and only if $P$ is not elementary abelian.

\noindent \eqref{lemma:(1,P)contractD}
There are functors 
  \begin{equation*}
    \xymatrix@1{
         \cat OP{[1,P)}\ar@<.75ex>[r]^-L &
        \cat OP{[\Phi(P),P)}  \ar@<.75ex>[l]^-R & 
        \cat O{P/\Phi(P)}{[1,P/\Phi(P))} \ar[l]_{\cong}^U }     
  \end{equation*}
  where $R$ and $L$ are adjoint functors and $U$ is an iso\m . The
  functors $R$ and $L$ are given by $QL=Q\Phi(P)$ and $QR = Q$ for $Q
  \leq P$.
% \begin{equation*}
%   \cat O{P/\Phi(P)}{[1,P/\Phi(P))} =
%   \cat OP{[\Phi(P),P)} \subset \cat OP{[1,P)}
% \end{equation*}
  The category in the middle, $\cat OP{[\Phi(P),P)}$, is isomorphic to
  the category $\cat O{P/\Phi(P)}{[1,P/\Phi(P))}$. To see this,
  observe that all supergroups of the Frattini subgroup $\Phi(P)$ are normal, so that $\cat OP{}(Q_1,Q_2) = P/Q_2 =
  \frac{P/\Phi(P)}{Q_2/\Phi(P)} = \cat
  O{P/\Phi(P)}{}(Q_1/\Phi(P),Q_2/\Phi(P))$ when $Q_1$ and $Q_2$ both
  contain $\Phi(P)$. 
%%%%%%%%%%%%%%%%%%%%%%%%%%%%%%
% Since the Frattini quotient of $P$ is cyclic
%   precisely when $P$ itself is cyclic \cite[Chp 5, Corollary
%   1.2]{gorenstein68}, we see that the claim of
%   \eqref{lemma:(1,P)contractD} is that
% \begin{equation*}
%   \text{$\cat O{E_n(p)}{[1,E_n(p))}$ is noncontractible} \iff n=1
% \end{equation*}
% where $E_n(p)$ is the elementary abelian $p$-group of rank $n$.  If
% $n=1$, then $E_1(p) = C_p$ is cyclic of order $p$, and $\cat
% O{E_1(p)}{[1,E_1(p))}$, containing a single object, is a group, namely
% $C_p$ again, which is noncontractible. However, when $n=2$ so that
% $E_2(p) = V$ is Klein's $4$-group, it seems that $\cat OV{[1,V)}$ is
% still noncontractible (Example~\ref{exmp:OCpn}). 
%%%%%%%%%%%%%%%%%%%%%%%%%%%%%%%%%%%%%%%%%

% The category to the left is $\mathcal{V}_n(p)$ (Definition~\ref{defn:Vn})
% where $n$ is the rank of the Frattini quotient $P/\Phi(P)$.  
% This category is not contractible if and only if $n=1$ if and only if
% $P$ is cyclic (Proposition~\ref{prop:Vncontract}).

  \noindent \eqref{lemma:(1,P)contractE} If $P$ is elementary abelian,
  then $\catt FP{} = \cat SP{}$ and $\catt FP{(1,P)} = \cat
  SP{(1,P)}$, is noncontractible by \eqref{lemma:(1,P)contractA}. If
  $P$ is not elementary abelian, the Frattini subgroup $\Phi(P)$ is a
  nontrivial normal subgroup and so is its intersection with the
  center $Z(P)$ of $P$ \cite[5.2.1]{robinson:groups}. There are
  adjoint equivalences of categories
\begin{equation*}
  \xymatrix@1{
    {\catt FP{(1,P)}} \ar@<.75ex>[r]^-{L} & 
    {\catt FP{[\Phi(P) \cap Z(P),P)}} \ar@<.75ex>[l]^-{R}
    \ar@<.75ex>[r]^-{R} &
    {\catt FP{[1,P)}} \ar@<.75ex>[l]^-{L} }  
\end{equation*}
where $QL=Q\Phi(P)$ and $QR = Q$ for $Q \lneqq P$.  The category to the
right, $\catt FP{[1,P)}$, is contractible because it has the trivial
group as an initial object. The category to the left, $\catt
FP{(1,P)}$, is therefore also contractible.  
\end{proof}

One might be led by
Lemma~\ref{lemma:(1,P)contract}.\eqref{lemma:(1,P)contractB} to
suspect that, for any nonidentity $p$-group $P$,
\begin{equation*}
  \text{$\cat OP{[1,P)}$ is noncontractible} \Longrightarrow
  \text{$P$ is cyclic}
\end{equation*}
or, equivalently, for any nonidentity elementary abelian
$p$-group $V$,
\begin{equation*}
  \text{$\cat OV{[1,V)}$ is noncontractible} \Longrightarrow
  \mathrm{rank}(V)=1
\end{equation*}
To see that these two statements are equivalent, recall that the
Frattini quotient of $P$ is cyclic precisely when $P$ itself is cyclic
\cite[Chp 5, Corollary~1.2]{gorenstein68} and use
Lemma~\ref{lemma:(1,P)contract}.\eqref{lemma:(1,P)contractD}. However,
Example~\ref{exmp:OCpn} demonstrates that these statements are false.

\begin{exmp}\label{exmp:OCpn}
  Let $V = C_p^r$ be the elementary abelian $p$-group of rank $r
  \geq 1$.  The objects of the category $\cat O{V}{[1,V)}$
  are the proper subgroups of $V$, and the set of \m s from $H
  \lneqq V$ to $K \lneqq V$ is
  \begin{equation*}
    \cat O{V}{[1,V)}(H,K) =
    \begin{cases}
      V/K & \text{if $H \leq K$} \\
      \emptyset & \text{otherwise}
    \end{cases}
  \end{equation*}
  with composition in this category induced from composition in the
  abelian group $V$.

% \begin{prop}\label{prop:Vncontract}\mynote{May be false!}
%   $\text{$\mathcal{V}_n(p)$ is not contractible} \iff n=1$
% \end{prop}
% \begin{proof}
%   The one-object category $\mathcal{V}_1(p)$ is the cyclic group of
%   order $p$ which is not contractible. 
% \end{proof}

%In fact the obvious functor $(1,P) \to \frac{(1,P)}{P}$ is a homotopy
%equivalence: If $P$ is elementary abelian, the two posets are
%identical, and if not, they are contractible.

  If the rank $r=1$, then the category $\cat O{V}{[1,V)} = \cat
  O{V}{\{1\}}$ is the cyclic group $V$, which is not contractible.

  Let us now explore the category $\cat O{V}{[1,V)}$ in the case where the
  rank $r > 1$. There is an obvious functor
   \begin{equation*}
     \pi \colon \cat OV{[1,V)} \to \cat SV{[1,V)}
   \end{equation*}
   to the poset of proper subgroups of $V$. For any proper subgroup
   $K$ of $V$, the $\pi$-slice over $K$ is $\pi/K = \cat OV{[1,K]}$.
   There is an adjunction
  \begin{equation*}
    \xymatrix@1{ {\cat OV{[1,K]}} \ar@<.75ex>[r]^-{L} &
        {\cat OV{\{K\}}} \ar@<.75ex>[l]^-{R}}, \qquad
      RL  \stackrel{\varepsilon}{=} 1_{\cat OV{\{K\}}}, \qquad
      1_{\cat OV{[1,K]}} \stackrel{\eta}{\Rightarrow} LR,
  \end{equation*}
  where $HL=K$ and $KR=K$. The functor $R$ includes the full
  subcategory of $\cat OV{}$ with $K$ as its only object into the full
  subcategory of all subgroups of $K$. The functor $L$ is the
  projection $\cat OV{}(H_1,H_2) = V/H_2 \to V/K = \cat OV{}(K,K)$,
  $H_1 \leq H_2 \leq K$.  Thus the category $\cat OV{[1,K]}$ is
  homotopy equivalent to the category $\cat OV{\{K\}}$ which is the
  group $V/K$.  
%   \mynote{Is this an argument?  Now $\pi$ projects the
%     category $\cat OV{[1,V)}$ onto the contractible category $ \cat
%     SV{[1,V)}$ with initial element $1$. If $\cat OV{[1,V)}$ were
%     contractible, then $\pi$ would be a homotopy equivalence, and the
%     \lq fibres\rq\ $\pi/K$ would be contractible - but they are not.}
%%%%%%%%%%%%%%%%%%%%%%
  The composite functor spectral sequence \cite[pp
  155--157]{gabriel_zisman67} \cite[Proof of Proposition
  2.3]{dw:codec} %% also in some paper by
                                            %% Dwyer-Wilkerson
  \begin{equation}\label{eq:ss}
    E^2_{st} = H_s(\cat SV{[1,V)} ; H_t(V/\bullet;\F_p)) \Longrightarrow 
    H_{s+t}(\cat OV{[1,V)};\F_p)
  \end{equation}
  associated to the functor $\pi$ provides information about the
  homology groups of the category $\cat OV{[1,V)}$. Here, we write
  $H_s(\cat SV{[1,V)}; H_t(V/\bullet))$ for the $s$th left derived of
  the functor $\colim H_t(V/\bullet)$. In concrete terms, these groups
  are the homology groups of the
  normalized chain complex \cite[Theorem VIII.6.1]{maclane} of the
  simplicial abelian group $\coprod_* H_t(V/\bullet)$
  \cite[XII.5.5]{bk},
  \begin{equation*}
      0  \leftarrow
      \bigoplus_{0 \leq L_0 < V} H_t(V/L_0) \xleftarrow{\partial_1}  
      \bigoplus_{0 \leq L_0 <L_1 < V} H_t(V/L_0) \xleftarrow{\partial_2}  
      \cdots \xleftarrow{\partial_s}
      \bigoplus_{0 \leq L_0 <L_1 \cdots < L_s < V} H_t(V/L_0)
      \xleftarrow{\partial_{s+1}} \cdots
  \end{equation*}
  with boundary homo\m\ $\partial_s$ is defined by deleting single
  entries of the $s$-flag $L_0 < L_1 < \cdots < L_s$ and applying
  $H_t(V/L_0) \to H_t(V/L_1)$ in the case of deletion of the first entry.
  This chain complex is trivial in degrees $> r-1$ so that the
  spectral sequence \eqref{eq:ss} is concentrated in the vertical band
  $0 \leq s \leq r-1$.

   Take $r=2$ and $p=2$ and consider the category $\cat OV{[1,V)}$ where
   $V$ is the Klein $4$-group.  The objects of $\cat
   OV{[1,V)}$ are the identity subgroup, $\{0\}$, and three subgroups,
   $L_1$, $L_2$, and $L_3$, of order $2$. The category $\cat
   OV{[1,V)}$ is
  \begin{equation*}
    \xymatrix{
      L_1 \ar@(ul,ur)^{V/L_1} & 
      L_2 \ar@(ul,ur)^{V/L_2} &
      L_3 \ar@(ul,ur)^{V/L_3} \\
      & {\{0\}} 
      \ar[ul]^-{V/L_1} \ar[u]_(.65){V/L_2} \ar[ur]_-{V/L_3}
      \ar@(dl,dr)_{V} } \qquad
    \zeta =
    \begin{pmatrix}
      4 & 2 & 2 & 2 \\ 0 & 2 & 0 & 0 \\
      0 & 0 & 2 & 0 \\ 0 & 0 & 0 & 2
    \end{pmatrix} \qquad
    k_\bullet^{\cat O{}{}} =
      (1/4,  1/4 , 1/4 , 1/4) \qquad \chi(\cat OV{[1,V)}) = 1
  \end{equation*}
  with composition induced from addition in the abelian group $V$.
  The first quadrant spectral sequence \eqref{eq:ss} is concentrated
  on the two vertical lines $s=0$ and $s=1$ so that all differentials
  are trivial. The groups $E^2_{0t}=E^\infty_{0t}$ and
  $E^2_{1t}=E^\infty_{1t}$ are the homology groups of the normalized
  simplicial replacement chain complex
  \begin{equation*}%\label{eq:cochains}
    \cdots \leftarrow 0 \leftarrow
    H_t(V/0) \oplus H_t(V/L_1) \oplus H_t(V/L_2) \oplus H_t(V/L_3)
    \leftarrow H_t(V/0) \oplus  H_t(V/0) \oplus  H_t(V/0) \leftarrow 0 \to
    \leftarrow 
  \end{equation*}
  concentrated in degrees $0$ and $1$. Since $H_t(V/0)$ has dimension
  $t+1$, $H_t(V/L_i)$, $i=1,2,3$, is $1$-dimensional, the term
  $E^\infty_{1,t}$ has dimension at least $2t-1$, and consequently
  $\dim_{\F_2} H_{t+1}(\cat OV{[1,V)};\F_2) \geq 2t-1$ for all degrees
  $t \geq 1$.
% (An explicit computation using Magma
%  \cite{magma} shows $H_1(\cat OV{[1,V)};\F_2) = 0$, $H_2(\cat
%  OV{[1,V)};\F_2) = \F_2$.)  
%  %%  /home/moller/projects/euler/magma/OV.prg

  The above argument is easily seen to work for any prime $p$ and we
  conclude that $\dim_{\F_p} H_{t+1}(\cat OV{[1,V)};\F_p) \geq pt-1$
  for all degrees $t \geq 1$ when the rank $r=2$. Thus $\cat
  OV{[1,V)}$ is noncontractible when $V$ has rank $r=2$.
%   %%%%%%%%%%%%%%%%%
%   We are going to use \cite[Theorem 13.2]{jmm:ndet} to compute the
%   cohomology groups of $\cat OV{[1,V)}$.  For every integer $t \geq 0$
%   there is a cochain complex concentrated in degrees $0$ and $1$
%   \begin{equation}\label{eq:cochains}
%     \cdots \to 0 \to
%     H^t(V) \times H^t(V/L_1) \times H^t(V/L_2) \times H^t(V/L_3)
%     \to H^t(V) \times  H^t(V) \times  H^t(V) \to 0 \to \cdots
%   \end{equation}
%   derived from the category $\cat OV{[1,V)}$ \cite[(13.1)]{jmm:ndet}.
%   Let $E_2^{0t}$ be the cohomology in degree $0$ and $E_2^{1t}$ the
%   cohomology in degree $1$.  Using the constant functor $\F_2$ for
%   coefficients, $H^t(V)$ has dimension $t+1$, $H^t(V/L_j)$, $1 \leq j
%   \leq 3$, dimension $1$, and the cochain complex \eqref{eq:cochains}
%   is of dimension $t+4$ in degree $0$ and $3t+3$ in degree $1$.  There
%   is a short exact sequence
%   \begin{equation*}
%     0 \to E_2^{11} \to H^2(\cat OV{[1,V)}) \to E_2^{02} \to 0
%   \end{equation*}
%   Taking $t=1$, we see that $E_2^{11}$ and $H^2(\cat OV{[1,V)};\F_2)$
%   are nontrivial.
%%%%%%%%%%%%%%%5
%    Alternatively, there is a long exact sequence
%    \begin{equation*}
%       \cdots \to H^t(\cat OV{[1,V)}) \to
%     H^t(V) \times H^t(V/L_1) \times H^t(V/L_2) \times H^t(V/L_3)
%     \to H^t(V) \times  H^t(V) \times  H^t(V) 
%     \to H^{t+1}(\cat OV{[1,V)}) \to \cdots
%    \end{equation*}
%    for the cohomology of $\cat OV{[1,V)}$.
\end{exmp}

Here are few remarks about the spectral sequence \eqref{eq:ss} for
arbitrary prime $p$ and rank $r \geq 2$.  When $t=0$, $E^2_{s0} =
H_s(\cat SV{[1,V)};\F_p)$, so that $E^2_{00} = \F_p$ and $E^2_{s0}=0$
for $s>0$, as $\cat SV{[1,V)}$ is contractible.  When $t>0$, we
conjecture, based on computer calculations, that $E^2_{st}=0$ except
for $s=r-1$. We have not been able to prove this conjecture.

\section{Brown posets and transporter categories}
\label{sec:brownposet}

We now begin the process of proving the results summarized in Theorem~\ref{thm:summary}, which will take up the remainder of the paper.

Let $G$ be a finite group of order divisible by $p$ and $\cat SG{}$
the poset of $p$-subgroups of $G$. The Brown poset for $G$ is the
subposet $\cat SG* = \cat SG{(1,G]}$ of nonidentity $p$-subgroups of
$G$.  We  show that the homotopy type of $\cat SG*$ is determined by either the elementary abelian $p$-subgoups of $G$, or the $G$-radical subgroups of $G$, as well as showing that the full subcategory of $p$-selfcentralizing  subgroups of $G$ has its homotopy determined by the $G$-radical, $p$-selfcentralizing  subgroups. The results of this section are not new, but they provide the template of our argument, which we outline now:

  For each claim of Theorem~\ref{thm:SG}, we break the proof into two separate parts:  The computation of the Euler characteristic of the general (co)slice of the larger category,  and actual proof of homotopy equivalence through an application of Bouc's Theorem~\ref{lemma:boucEI2}.  These parts are labelled {\bf[EC]} and {\bf[HE]}, respectively.  The truth of the result is shown in the second part, whereas the Euler characteristics calculation is not, strictly speaking, necessary for the proof of the Theorem.  Instead, it is  offered as a moral  argument as to why we should expect the result to be true, following Remark~\ref{rmk:BoucInterpretation}.  
  
	Let us consider as a toy  example Part (a) of Theorem~\ref{thm:SG}, which is Quillen's result \cite{quillen78} that the homotopy type of $\cat SG*$ is determined by the subposet $\cat SG{*+\mathrm{eab}}$ of elementary abelian $p$-subgroups.   By Theorem~\ref{lemma:boucEI2}, we must find those objects $K$ of $\cat SG*$ whose proper slice categories $\cat SG*//K$ are not contractible.  We will ultimately see that says that the objects are precisely the elementary abelian $p$-subgroups of $G$ (the ``second part'' of the proof), but first we pretend not to know this and ask what sort of subgroups we should consider.  A good first guess would be those objects whose proper slice categories have nonzero reduced Euler characteristic, as those proper slice categories are necessarily noncontractible.  This does not guarantee that all other objects have contractible proper slice categories, but with a little more work, this turns out to be the case.  This basic chain of reasoning will be repeated for all of the similar results that follow.

% \noindent\underline{FACTS}:
%  \begin{itemize}
%  \item All \m s in  $\cat SG{}$ are mono\m s and epi\m s
%  \item The categories $H//\cat SG*$ and $\cat SG*//K$ are thin (they are
%    posets)
%  \item The \we\ for $\cat SG{}$, $k^H = -\rchi(\cat S{\cat OG{}(H)}*)$,
%    vanishes off the $p$-radical $p$-subgroups of $G$
%    \cite[Theorem~1.3.(1)]{jmm_mwj:2010}
%  \item The co\we\ for $\cat SG{}$, $k_K = -\rchi(\cat SK{(1,K)})$ vanishes
%    off the elementary $p$-subgroups \cite[Theorem~1.1.(1)]{jmm_mwj:2010}
%   \end{itemize}

  \begin{thm} \label{thm:SG} \cite{bouc84a} \cite{quillen78} 
    The following inclusions are homotopy equivalences:
    \[
    \mathbf{(a)}\ \cat SG{*+\mathrm{eab}} \hookrightarrow
    \cat SG*\qquad\qquad
    \mathbf{(b)}\ \cat SG{*+\mathrm{rad}} \hookrightarrow \cat SG*\qquad\qquad
    \mathbf{(c)}\ \cat SG{\mathrm{sfc}+\mathrm{rad}} \hookrightarrow \cat
    SG{\mathrm{sfc}}
    \]
  \end{thm}
  \begin{proof}  We now flesh out the details of the above paragraph:
  \begin{itemize}
  \item[{\bf(a)}]
  
    {\bf[EC]} The coweighting for $\cat SG*$ can be expressed in two different ways by \cite[Theorem~1.1.(1)]{jmm_mwj:2010} and
    Theorem~\ref{lemma:underwt},
    \begin{equation*}
       -\rchi(\cat SK{(1,K)})  = k^K = -\rchi(\cat SG*//K),
    \end{equation*}
    so that  the categories $\cat SG*//K$ and $\cat SK{(1,K)}$ have
    identical \Euc s. By Lemma~\ref{lemma:(1,P)contract}.\eqref{lemma:(1,P)contractA}, $\rchi\left(\cat SK{(1,K)}\right)=0$ unless $K$ is elementary abelian.  This suggests that the class of subgroups with noncontractible proper slice categories is precisely the elementary abelian $p$-subgroups of $G$.
    
    \noindent{\bf[HE]} Indeed, the categories  $\cat SG*//K$ and $\cat SK{(1,K)}$ not only have the same Euler characteristics, they are themselves identical! 
    Lemma~\ref{lemma:(1,P)contract}.\eqref{lemma:(1,P)contractA} then gives us more information:
    \begin{equation*}
      \mathrm{supp}(\cat SG*//\bullet) = \{k \in \Ob{\cat SG*} \mid 
      \text{$K$ is elementary abelian} \} = \Ob{\cat SG{*+\mathrm{eab}}}
    \end{equation*}
     and Bouc's Theorem~\ref{lemma:boucEI2} shows that the inclusion of
    $\cat SG{*+\mathrm{eab}}$ into $\cat SG*$ is a homotopy equivalence.

  \item[{\bf(b)}]
    {\bf[EC]} The weighting for $\cat SG*$  can be expressed in two different ways by \cite[Theorem~1.3.(1)]{jmm_mwj:2010} and
    Theorem~\ref{lemma:underwt}:
    \begin{equation*}
       -\rchi(\cat S{\cat OG{}(H)}*)  = k^H_{\cat S{}{}} 
       = -\rchi(H//\cat SG*).
    \end{equation*}
    In particular,  the categories $H//\cat SG*$ and $\cat S{\cat
      OG{}(H)}*$ have identical \Euc s. By Property~\eqref{eq:Grad}, if $\cat S{\cat OG{}(H)}*$ is not contractible, then $H$ must be $G$-radical.  Therefore the class of subgroups whose proper coslice category has nonzero Euler characteristic is contained in the class of $G$-radical subgroups.

   \noindent{\bf[HE]} We show that this equality of reduce Euler characteristics reflects a homotopy equivalence $H//\cat SG*\simeq\cat S{\cat OG{}(H)}*$.  For any nonidentity
    $p$-subgroup $H$ of $G$, there are functors
    \begin{equation*}
      \xymatrix@1{  H/\cat SG* \ar@<.75ex>[r]^-{r_H} &
        {\cat S{\cat OG{}(H)}{}} \ar@<.75ex>[l]^-{i_H} }, \qquad
     \xymatrix@1{  H//\cat SG* \ar@<.75ex>[r]^-{r_H} &
        {\cat S{\cat OG{}(H)}{*}} \ar@<.75ex>[l]^-{i_H} }
    \end{equation*}
%%%%%%%%%%%%%%%%%%%%%%%%%%
 %    \begin{equation*}
%       r_H \colon  H/\cat SG* \to    \cat S{\cat OG{}(H)}{}, \qquad
%       r_H \colon  H//\cat SG* \to    \cat S{\cat OG{}(H)}{*}
%     \end{equation*}
%%%%%%%%%%%%%%%%%%%%%%
   %  \begin{equation*}
%       \xymatrix@1{ H/\cat SG* \ar[r]^-R &  
%       \cat S{\cat OG{}(H)}{}}, \qquad
%       \xymatrix@1{ H//\cat SG* \ar[r]^-R &  
%       \cat S{\cat OG{}(H)}{*}}
%     \end{equation*}
    given by $Kr_H = N_K(H)/H$ for all $p$-supergroups $K$ of $H$ and
    $\overline{K}i_H = K$ when $\overline{K}=K/H$ and $H \leq K \leq
    N_G(H)$ (\cite[Lemma 6.1]{quillen78}). The composite functor
    $i_Hr_H$ is the identity of $\cat S{\cat OG{}(H)}{}$ and there is
    a natural transformation from $r_Hi_H:K\mapsto N_K(H)$ to the identity
    functor of $H/\cat SG*$. This shows that these functors are
    homotopy equivalences of categories.
%%%%%%%%%%%%%%%%%%%%
  %   For any $p$-subgroup $\overline{L} = L/H$ of $\cat OG{}(H) =
%     N_G(H)/H$, $\overline{L}/r_H$ is the full subposet of $\cat SG{}$
%     consisting of all $p$-supergroups $K$ of $H$ with $L \leq N_K(H)$.
%     Thus $L$ is initial in $\overline{L}/r_H$. By Quillen's
%     \cite[Theorem~A]{quillen73}, the functors $r_H$ are homotopy
%     equivalences. (In fact they are equivalences \mynote{adjoint
%       equivalence?} of categories as there are obvious functors in the
%     other direction \cite[Proposition 6.1]{quillen78}.)
%%%%%%%%%%%%%%%%%%%%%%%%%%%%%%%
    By Property~\eqref{eq:Grad},
    \begin{multline*}
      \hspace{\leftmargin}\mathrm{supp}(\bullet//\cat SG*) =
      \{ H \in \Ob{\cat SG*} \mid \text{$\cat S{\cat OG{}(H)}*$ is
        noncontractible} \} \\
        \subseteq
          \{ H \in \Ob{\cat SG*} \mid \text{$H$ is $G$-radical} \}
          = \Ob{\cat SG{*+\mathrm{rad}}}
    \end{multline*}
    and Bouc's Theorem~\ref{lemma:boucEI2} shows that the inclusion of
    $\cat SG{*+\mathrm{rad}}$ into $\cat SG*$ is a homotopy equivalence.

\item[{\bf(c)}]    Since any $p$-supergroup of a $p$-selfcentralizing $G$-subgroup is
    itself $p$-selfcentralizing, $H//\cat SG{\mathrm{sfc}}~=~H//\cat
    SG*$ for any $p$-selfcentralizing subgroup $H$ of $G$. The result then follows from Part (b).
    \end{itemize}
  \end{proof}

  \begin{exmp}
    If $G = C_2 \times \Sigma_3$ and $p=2$, then $\cat
    SG{\mathrm{sfc}}$ is a discrete poset consisting of the $3$ \syl
    2s, while $\cat SG*$ is contractible since $O_2G = C_2$ is
    nontrivial. Thus the inclusion $\cat SG{\mathrm{sfc}} \to \cat
    SG*$ is {\em not\/} a homotopy equivalence.
  \end{exmp}

  The following proposition points out that the largest normal
  $p$-subgroup is the smallest $G$-radical $p$-subgroup.  It implies
  that the poset $\cat SG{*+\mathrm{rad}}$ has a smallest element in the case
  $O_pG$ is nontrivial.  In light of Theorem~\ref{thm:SG}(b), this could be thought of as the essential ingredient that goes into Quillen's Property~\ref{eq:SK*contract}. (We thank Andy Chermak for the proof.)

  \begin{prop}\label{prop:OpG}
    Any $G$-radical $p$-subgroup of $G$ contains the $G$-radical
    $p$-subgroup $O_pG$.
   %  $O_p(G)$ is the least element of the poset $\cat SG{\mathrm{rad}}$
%     of $G$-radical $p$-subgroups of $G$.
  \end{prop}
  \begin{proof}
    It is clear that $O_pG$ is a normal $G$-radical $p$-subgroup.
    Let $H$ be a $p$-subgroup of $G$ not containing $O_pG$.  The
    normalizer of $H$ in the $p$-subgroup $(O_pG)H$ is normal in
    $N_G(H)$ for any element of $G$ normalizing $H$ normalizes
    $(O_pG)H$. Since $N_{(O_pG)H}(H)$ is a normal $p$-subgroup of
    $N_G(H)$ strictly larger than $H$, the $p$-subgroup $H$ is not
    $G$-radical.
  \end{proof}

We close this section by moving from $p$-subgroup posets to more general EI-categories.  Let $\cat TG{}$ be the transporter category of $p$-subgroups of $G$.

  \begin{thm}\label{prop:TG}
    The following inclusions are homotopy equivalences:
    \[
    {\bf(a)}\ \cat TG{*+\mathrm{eab}} \hookrightarrow\cat TG*\qquad\qquad
    {\bf(b)}\ \cat TG{*+\mathrm{rad}} \hookrightarrow \cat TG*\qquad\qquad
    {\bf(c)}\ \cat TG{\mathrm{sfc}+\mathrm{rad}} \hookrightarrow \cat TG{\mathrm{sfc}}
    \]
  \end{thm}
  \begin{proof}
    Every morphism of $\cat TG*$ is both epi and mono, so it follows that the (co)slice categories of objects should be identifiable with the Brown posets of certain groups measuring relating those objects to $G$.  With this in mind, the argument follows that of Theorem~\ref{thm:SG} closely.
    \begin{itemize}
    \item[{\bf(a)}]     {\bf[EC]} The coweighting on $[\cat TG*]$  is computed in \cite[Theorem~1.1.(2)]{jmm_mwj:2010}.   Theorem~\ref{lemma:underwt} gives an alternate calculation of the coweighting in terms of Euler characteristics of proper slice categories:
     \begin{equation*}
       \frac{-\rchi(\cat SK{(1,K)})}{|\cat TG*(K)|}  = 
       k_{[K]}^{[\cat T{}{}]} 
       = \frac{-\rchi(\cat TG*//K)}{|\cat TG*(K)|}.
    \end{equation*}
      Lemma~\ref{lemma:(1,P)contract}.\eqref{lemma:(1,P)contractB} then implies that $\rchi(\cat TG*//K)$ is nonzero iff $K$ is  elementary abelian.
    
    \noindent{\bf[HE]}  In fact, Lemma~\ref{lemma:(1,P)contract} says more:  $\cat SK{(1,K)}$ is \emph{noncontractible} iff $K$ is elementary abelian.  Our goal is then to show that the equality of reduced Euler characteristics $\rchi(\cat SK{(1,K)})=\rchi(\cat TG*//K)$ reflects a homotopy equivalence $\cat SK{(1,K)}=\cat TG*//K$; once this has been accomplished, Theorem~\ref{lemma:boucEI2} will complete the result.
    
    There are functors
    
    \begin{equation*}
      \xymatrix@1{ \cat SK{(1,K]} \ar@<.75ex>[r]^-{r_K} &
        {\cat TG*/K} \ar@<.75ex>[l]^-{i_K} }, \qquad
      \xymatrix@1{ \cat SK{(1,K)} \ar@<.75ex>[r]^-{r_K} &
        {\cat TG*//K} \ar@<.75ex>[l]^-{i_K} }, 
    \end{equation*}
    
    given by $Hr_K=(H\xrightarrow{1} K)$ and $(H\xrightarrow g K)i_K=H^g$.  Clearly these are equivalences of categories, so we have our desired homotopy equivalence $\cat SK{(1,K)}=\cat TG*//K$ and the result is proved.
    
%    lead to homotopy equivalence $\cat TG*//K \simeq \cat SK{(1,K)}$
%    for any object $K$ of $\cat TG*$.
%    Lemma~\ref{lemma:(1,P)contract}.\eqref{lemma:(1,P)contractA} and
%    Bouc's Theorem~\ref{lemma:boucEI2} now imply that the inclusion
%    functor $\cat TG{*+\mathrm{eab}} \hookrightarrow \cat TG*$ is a
%    homotopy equivalence.

    \item[{\bf(b)}] {\bf[EC]} The weighting of $[\cat TG*]$ was computed in  \cite[Theorem~1.3.(2)]{jmm_mwj:2010}.  Comparing this to the alternate calculation of the weighting
    Theorem~\ref{lemma:underwt}, we have
     \begin{equation*}
       \frac{-\rchi(\cat S{\cat OG{}(H)}*)}{|\cat TG*(H)|}  = 
       k^{[H]}_{[\cat T{}{}]} 
       = \frac{-\rchi(H//\cat TG*)}{|\cat TG*(H)|}.
    \end{equation*}
     Property~\eqref{eq:Grad} implies that $\rchi(H//\cat TG*)\neq 0$ implies that $H$ is $G$-radical.
     
     \noindent{\bf[HE]}  If we can show that there is a homotopy equivalence $\cat S{\cat OG{}(H)}*\simeq H//\cat TG*$, the full strength of  Property~\eqref{eq:Grad} will yield $\mathrm{supp}(\bullet//\cat TG*)$ is contained in the class of $G$-radical subgroup, so Theorem~\ref{lemma:boucEI2} will give the result.  There are functors  
     
         \begin{equation*}
      \xymatrix@1{  H/\cat TG* \ar@<.75ex>[r]^-{r_H} &
        {\cat S{\cat OG{}(H)}{}} \ar@<.75ex>[l]^-{i_H} }, \qquad
     \xymatrix@1{  H//\cat TG* \ar@<.75ex>[r]^-{r_H} &
        {\cat S{\cat OG{}(H)}{*}} \ar@<.75ex>[l]^-{i_H} }
    \end{equation*}
    
    given by $(H\xrightarrow{g} K)r_H=N_K(H^g)^{g^{-1}}/H$ and $\overline K i_H=(H\xrightarrow{1}K)$ where $\overline K=K/H$ and we have $H~\leq~K~\leq~N_G(H)$.  Clearly $i_Hr_h=\mathrm{id}_{\cat S{\cat OG{}(H)}{*}}$, and we have a natural transformation $\eta:r_Hi_H~\Rightarrow~\mathrm{id}_{H//\cat TG*}$ induced by the inclusion $N_K(H^g)^{g^{-1}}\leq N_G(H)$.  Thus the two categories are homotopy equivalent, and the result is proved.

%    lead to homotopy equivalences $H//\cat TG* \simeq \cat S{\cat
%      OG{}(H)}*$ for any object $H$ of $\cat TG*$.
%    Property~\ref{eq:Grad} and Bouc's Theorem~\ref{lemma:boucEI2} now
%    imply that the inclusion functors $\cat TG{*+\mathrm{rad}}
%    \hookrightarrow \cat TG*$, $\cat TG{\mathrm{sfc}+\mathrm{rad}}
%    \hookrightarrow \cat TG{\mathrm{sfc}}$ are homotopy equivalences.

	\item[{\bf(c)}]  Follows from Part (b) and the observation that supergroups of $p$-selfcentralizing subgroups of $G$ are themselves $p$-selfcentralizing.
	 
   \end{itemize}

  \end{proof}

  Suppose that $\cat C{}{}$ is a small category and $X,Y \colon \cat
  C{}{} \to \mathrm{CAT}$ are functors with values in the category
  $\mathrm{CAT}$ of small categories. If there is a natural
  transformation from $X$ to $Y$ with components $X(c) \to Y(c)$, $c
  \in \Ob{\cat C{}{}}$, that are all homotopy equivalences, then the
  induced functor $\int_{\cat C{}{}} X \to \int_{\cat C{}{}} Y$ of
  Grothendieck constructions is a homotopy equivalence. This follows
  from Thomason's homotopy colimit theorem \cite{thomason79} and
  homotopy invariance of the homotopy colimit \cite[Ch.\ XII, \S4,
  Homotopy Lemma 4.2]{bk}.  As the inclusions of Theorem~\ref{thm:SG}
  are $G$-equivariant inclusions of $G$-categories and $\cat TG* =
  (\cat SG*)_{hG}$ etc. we obtain an alternative proof of
  Proposition~\ref{prop:TG}. Similarly, if $O_pG$ is nontrivial, there
  is a homotopy equivalence $G \hookrightarrow \cat
  TG{*+\mathrm{rad}}$ induced by the $G$-equivariant homotopy
  equivalence $\ast \hookrightarrow \cat SG{*+\mathrm{rad}}$ of
  Proposition~\ref{prop:OpG}.

  \section{Frobenius categories}
  \label{sec:frob}

  Let $P$ be a finite $p$-group and $\cat F{}{}$ a Frobenius
  $P$-category.   In this section we show that the homotopy type of $\cat F{}*$ is determined by the elementary abelian subgroups of $S$.  

We will need the following facts:
  \begin{itemize}
  \item All \m s in  $\cat F{}{}$ are mono\m s, which implies
  \item For any $K\leq P$, the categories $\cat F{}*/K$ and  $\cat F{}*//K$ are \emph{thin}, i.e., there is at most one morphism between any two objects
  \item The co\we\ for $\cat F{}*$ vanishes off the elementary abelian
    subgroups \cite[Theorem~7.5]{jmm_mwj:2010}
  \end{itemize}

\begin{thm}\label{prop:FeabtoF*II}
  The inclusion $\cat F{}{*+\mathrm{eab}} \to \cat F{}{*}$  is
  a homotopy equivalence. 
\end{thm}
\begin{proof}$\left.\right.$
  \begin{itemize}
  \item[{\bf[EC]}] We compute the coweighting on $[\cat F{}*]$ using both  \cite[Theorem~7.5]{jmm_mwj:2010} and Theorem~\ref{lemma:underwt}:
      \begin{equation*}
      \frac{-\rchi(\cat SK{(1,K)})}{|\cat F{}*(K)|} =
      k^{[\cat F{}*]}_{[K]} =
      \frac{-\rchi(\cat F{}*//K)}{|\cat F{}*(K)|} 
    \end{equation*}
     Therefore $\cat SK{(1,K)}$ and $\cat F{}*//K$ have identical \Euc
    s. By   Lemma~\ref{lemma:(1,P)contract}.\eqref{lemma:(1,P)contractA}, $\rchi(\cat F{}*//K)\neq 0$ implies $K$ is elementary abelian.
    
   \item[{\bf[HE]}] Indeed, there are functors
  \begin{equation*}
      \xymatrix@1{  \cat F{}*/K \ar@<.75ex>[r]^-{r_K} &
        {\cat S{K}{(1,K]}} \ar@<.75ex>[l]^-{i_K} }, \qquad
     \xymatrix@1{  \cat F{}*//K \ar@<.75ex>[r]^-{r_K} &
        {\cat SK{(1,K)}} \ar@<.75ex>[l]^-{i_K} }
    \end{equation*}
The functor $r_K$ takes $\varphi \in \cat F{}*(H,K)$ to its image
$H^\varphi$ in $K$. The functor $i_K$ takes $H \leq K$ to the
inclusion $H \hookrightarrow K$ of $H$ into $K$. Obviously, $i_Kr_K$
is the identity functor of $\cat SK{(1,K]}$, and there is a natural
transformation from the identity functor to the endofunctor $r_Ki_K:(H
\xrightarrow{\varphi} K)\mapsto (H^\varphi \hookrightarrow K)$ of
$\cat F{}*/K$. This shows that $r_K$ and $i_K$ are homotopy
equivalences between $\cat F{}*/K$ and $\cat SK{(1,K]}$. Their
restrictions are  homotopy
equivalences between $\cat F{}*//K$ and $\cat SK{(1,K)}$.
 %    For any subgroup $L \in (1,K]$, the category
%     $r_K/L$ is a full subcategory of the (thin) category $\cat F{}*/K$
%     with the inclusion $L \hookrightarrow K$ as terminal object.
%     According to Quillen's \cite[Theorem~A]{quillen73}, $r_K$ is a
%     homotopy equivalence. 
%%%%%%%%%%%%%%%%%%
   %  The functor $r_K$ restricts to a functor
%     \begin{equation}\label{eq:FKrK}
%         \func {r_K^*}{\cat F{}*//K}{\cat SK{(1,K)}}    %%{(1,K)}
%     \end{equation}
%     on the noniso\m\ objects of $\cat F{}*/K$. For any subgroup $L \in
%     (1,K)$, the category $r_K^*/L = r_K/L$ is contractible. Thus also
%     $r_K^*$ is a homotopy equivalence.
     By the full strength of
    Lemma~\ref{lemma:(1,P)contract}.\eqref{lemma:(1,P)contractA},
    \begin{equation*}
      \mathrm{supp}(\cat F{}*//\bullet) =  \Ob{\cat F{}{*+\mathrm{eab}}}
    \end{equation*}
    and Bouc's Theorem~\ref{lemma:boucEI2} shows that the inclusion of
    $\cat F{}{*+\mathrm{eab}}$ into $\cat F{}*$ is a homotopy
    equivalence.
%%%%%%%%%%%%%%%
  %   If $\cat F{}*//K \simeq \cat SK{(1,K)}$ is noncontractible, $K$ is
%     elementary abelian
%     (Lemma~\ref{lemma:(1,P)contract}.\eqref{lemma:(1,P)contractA}),
%     and Bouc theory (Lemma~\ref{lemma:boucEI2})
%     shows that the inclusion $\cat F{}{\mathrm{eab}} \to \cat F{}*$ is
%     a homotopy equivalence.
\end{itemize}
\end{proof}

In the course of the proof of Theorem~\ref{prop:FeabtoF*II} we
saw that the homotopy type of the category $\cat F{}{*}//K$ of
$\cat F{}*$-noniso\m s to $K$ depends only on $K$, not on $\cat
F{}{}$. This reflects the curious fact that the shape of the Frobenius $P$-category is able to detect some algebraic information of the underlying $p$-group.

% \mynote{Can $\cat F{}*$ be reconstructed from the categories
%   $\cat F{}*//K$?}

% \mynote{Is $\cat F{}{\mathrm{we}\neq 0} \to \cat F{}*$ a homotopy
%   equivalence?
% and $\cat
%   F{}{\mathrm{we} \neq 0 + \mathrm{sfc}} \to \cat F{}{\mathrm{sfc}}$?}

We know of no formula for the \we\ of a {\em general\/} Frobenius
category $\cat F{}{}$. There is an explicit formula in
\cite[Theorem~1.3.(3)]{jmm_mwj:2010} for the \we\ of the Frobenius
category $\cat FG{}$ { associated to a finite group\/} $G$, but we have not
been able to determine the support of this \we\ or describe the
categories $H//\cat FG*$.

  \section{Orbit categories}
  \label{sec:orbit}

  Let $G$ be a finite group of order divisible by $p$ and  $\cat
  OG{}$ the orbit category of $p$-subgroups of $G$. 

We will need the following facts:
 \begin{itemize}
 \item The trivial subgroup is not initial in $\cat OG{}$
 \item All \m s in  $\cat OG{}$ are epi\m s, therefore
 \item The categories $H/\cat OG*$ and $H//\cat OG*$ are thin
  \item The \we\ for $\cat OG{}$ vanishes off  the $G$-radical
  $p$-subgroups of $G$ \cite[Proposition 3.14]{jmm_mwj:2010}
  \item The co\we\ for   $\cat OG{}$ vanishes off the cyclic
    $p$-subgroups \cite[Theorem 4.1]{jmm_mwj:2010}
  \end{itemize}

  \begin{thm}\label{thm:GroupOrbitRadical}
    The following inclusions are homotopy equivalences:
     \[
     {\bf(a)}\ \cat OG{\mathrm{rad}} \hookrightarrow \cat OG{}\qquad\qquad
     {\bf(b)}\ \cat OG{*+\mathrm{rad}} \hookrightarrow \cat OG{*}\qquad\qquad
     {\bf(c)}\ \cat OG{\mathrm{sfc}+\mathrm{rad}} \hookrightarrow \cat
    OG{\mathrm{sfc}}
    \]
  \end{thm}
% /home/moller/projects/euler/magma/weFc.prg Fwtess(Sym(6),2) 
% The converse ibs not true Fwtess(Sym(8),2) but it seems that all
 % subgroups of negative weight and of order \geq something are
 % essential... 
  \begin{proof} 
  The setup for each claim is identical.
  \begin{itemize}
  \item[{\bf[EC]}] The two expressions for the weighting for $[\cat OG{}]$ from  \cite[Equation
    (3.15)]{jmm_mwj:2010} and Theorem~\ref{lemma:underwt} yield
 \begin{equation*}
   \frac{-\rchi(\cat S{\cat OG{}(H)}*)}{|\cat OG{}(H)|} =
   k^{[H]}_{[\cat OG{}]} = 
   \frac{-\rchi(H // \cat OG{})}{|\cat OG{}(H)|},
  \end{equation*}
  so that $\cat S{\cat OG{}(H)}*$ and $H // \cat OG{}$ have
  identical \Euc s.  Since $\rchi(\cat S{\cat OG{}(H)}*)\neq 0$ implies $H$ is $G$-radical (Property~\eqref{eq:Grad}), each claim is at least plausible.
  
  \item[{\bf[HE]}] We show that the equality of reduced Euler characteristics reflects a homotopy equivalence $\cat S{\cat OG{}(H)}*\simeq H//\cat OG{}$; the result will then follow from the contractibility of $\cat S{\cat OG{}(H)}*$ by Property~\eqref{eq:Grad} and Theorem~\ref{lemma:boucEI2}.  For any nonidentity
    $p$-subgroup $H$ of $G$, there are functors
   \begin{equation*} %%\label{eq:rHOG}
        r_H \colon H/\cat OG{} \to \cat S{\cat OG{}(H)}{}, \qquad 
        r_H \colon H//\cat OG{} \to \cat S{\cat OG{}(H)}{*}.
    \end{equation*} 
    The functor $r_H$ takes $gK \in \cat OG{}(H,K) = N_G(H,K)/K$ to
    the subgroup $N_{{}^gK}(H)/H$ of $\cat OG{}(H) = N_G(H)/H$. Let
    $L$ be a $p$-subgroup such that $H \leq L \leq N_G(H)$ and let
    $\overline{L} = L/H$ be the image of $L$ in $N_G(H)/H = \cat
    OG{}(H)$. The category $\overline{L}/r_H$ is the full subcategory
    of $\cat OG{}/H$ generated by all \m s $gK \in \cat OG{}(H,K)$
    such that $L \leq N_{{}^gK}(H)$. The inclusion of $H$ into $L$ is
    an object of $\overline{L}/r_H$ as $L=N_L(H)$.  Note that the \m\
    $gK \colon H \to K$ extends to a \m\ $gK \colon L \to K$ because
    $L^g \leq N_{{}^gK}(H)^g = N_K(H^g) \leq K$. There is thus a \m\
    \begin{equation*}
      \xymatrix{
        & H \ar@{_(->}[dl] \ar[dr]^{gK} \\
        L \ar[rr]_{gK} && K }
    \end{equation*}
    in $\overline{L}/r_H$. This shows that the inclusion $H
    \hookrightarrow L$ is an initial object of $\overline{L}/r_H$. By
    Quillen's Theorem~A (Theorem~\ref{thm:QuillenA}), the functor
    $r_H$ is a homotopy equivalence from $H/\cat OG{}$ to $\cat S{\cat
      OG{}(H)}{}$.  The same argument shows that $r_H$ restricts to a
    homotopy equivalence from $H//\cat OG{}$ to $\cat S{\cat
      OG{}(H)}*$.  Thus $\mathrm{supp}(\bullet//\cat OG{})~\subset~\Ob{\cat OG{\mathrm{rad}}}$ 
     and Part (a) is proved.
     
    Since $\cat OG*$ and $\cat OG{\mathrm{sfc}}$ are left ideals in
    $\cat OG{}$, $H//\cat OG{*} = H// \cat OG{}$ and $H//\cat
    OG{\mathrm{sfc}} = H// \cat OG{}$ for any nonidentity,
    respectively, $p$-selfcentralizing  subgroup $H$ of $G$. By
    property \eqref{eq:Grad},
    \begin{equation*}
      \mathrm{supp}(\bullet // \cat OG{*}) \subset
      \Ob{\cat OG{*+\mathrm{rad}}}, \qquad
      \mathrm{supp}(\bullet // \cat OG{\mathrm{sfc}}) \subset
      \Ob{\cat OG{\mathrm{sfc}+\mathrm{rad}}}
    \end{equation*}
    proving (b) and (c).
    \end{itemize}
      \end{proof} 

   %%%%%%%%%%%%%%%%%%
   %  For
%     every object $K$ of $\cat OG{}$, $\cat OG{}/K$ contains $\cat
%     OG{}//K$ as a right ideal. Thus any co\we\ on $\cat OG{}/K$
%     restricts to a co\we\ on $\cat OG{}//K$ \cite[Remark
%     2.6]{jmm_mwj:2010}. Since the co\we\ for $\cat OG{}$ vanishes off
%     the cyclic subgroups, we want to show that $\cat OG{}//K$ is
%     contractible when $K$ is not cyclic. From
%     Theorem~\ref{lemma:underwt} we know that the reduced \Euc\ 
%     $\rchi(\cat OG{}//K)=0$. 

    It seems that there should be a dual result to Theorem~\ref{thm:GroupOrbitRadical} involving certain ``small'' subgroups in place of the ``large'' $G$-radical class.  More precisely, there should be a theorem whose proof uses slices in place of the coslices of the previous argument.  The relevant class of subgroups to consider would then be those contained in $\mathrm{supp}(\cat OG{}//\bullet)$.  However, we cannot identify this class of subgroups at this point, and indeed experimental evidence leads us to conjecture that \emph{all} $p$-subgroups will necessarily be contained in the support.  If this conjecture holds,  the  dual theorem would reduce to the tautology $\cat OG{}\simeq\cat OG{}$, which would not be particularly enlightening.

  \section{Exterior quotients of Frobenius categories}
  \label{sec:tildeF}

  Let $P$ be a nonidentity finite $p$-group, $\cat F{}{}$ a Frobenius
  $P$-category, and $\catt F{}{}$ the exterior quotient of $\cat
  F{}{}$ \cite[1.3, 2.6, 4.8]{puig09}.  In this section we examine the homotopy types of $\catt F{}*$ and  $\catt F{}{\mathrm{sfc}}$.

 % Let   $\catt F{}{\mathrm{rad}}$ be the full subcategory of  $\catt
 % F{}{*}$ generated by the $\cat F{}{}$-radical subgroups. 

%/home/moller/projects/euler/magma/undercattildeF.prg use undercattF

We begin with $\catt F{}*$, searching for a class of ``small'' subgroups that control the homotopy type.  For our Euler characteristic intuition-building, the essential fact here is that the coweighting for $[\catt F{}*]$ vanishes off of the elementary abelian subgroups by \cite[Theorem 7.7]{jmm_mwj:2010}.
\begin{thm}\label{thm:OrbitCatElemAb}
The inclusion $\catt F{}{*+\mathrm{eab}} \hookrightarrow \catt F{}{*}$ is a homotopy equivalence.
\end{thm}
\begin{proof}$\left.\right.$
\begin{itemize}
\item[{\bf[EC]}]
    Comparing the reduced Euler characteristic expression  for the coweighting of $[\catt F{}*]$ from Theorem~\ref{lemma:underwt} to  \cite[Theorem 7.7]{jmm_mwj:2010}  yields
    \begin{equation*}
       \frac{-\rchi(\catt FK{(1,K)})}{|\catt F{}*(K)|} =
        k_{[K]}^{[\catt F{}*]} = 
        \frac{-\rchi(\catt F{}*//K)}{|\catt F{}*(K)|}.
    \end{equation*}
    Therefore $\catt FK{(1,K)}$ and $\catt F{}*//K$ have identical
    \Euc s for any object $K$ of $\catt F{}*$.  By
     Lemma~\ref{lemma:(1,P)contract}.\eqref{lemma:(1,P)contractB} and \eqref{lemma:(1,P)contractE}, $\rchi(\catt F{}*//K)$ can only be nonzero if $K$ is elementary abelian.
    
\item[{\bf[HE]}] In fact, there are
    equivalences of categories
    \begin{equation*}
      i_K \colon \catt FK{(1,K]} \to \catt F{}*/K, \qquad
      i_K \colon \catt FK{(1,K)} \to \catt F{}*//K
    \end{equation*}
    On an object $H\leq K$, we have $Hi_K = [\iota^H_K] \in \catt F{}*(H,K)$ is the class
    of the inclusion $\iota^H_K \in \cat F{}*(H,K)$ of $H$ into $K$. Observe that there is an obvious
    identification of \m\ sets
    \begin{equation*}
      \catt FK*(H_1,H_2) = (\catt F{}{*}/K)(H_1i_K, H_2i_K),
    \end{equation*}
    which defines the effect of  $i_K$ on \m\ sets. Thus  $i_K$ is full and faithful. It is also easily seen to be essentially
    surjective on objects, hence an equivalence of
    categories. %This applies to both above versions of $i_K$.
     
    Combining the homotopy equivalence $\catt FK{(1,K)}\simeq\cat F{}*//K$ with Lemma~\ref{lemma:(1,P)contract}.\eqref{lemma:(1,P)contractE}, we have
    \begin{equation*}
      \mathrm{supp}(\catt F{}*//\bullet) =  \Ob{\catt F{}{*+\mathrm{eab}}}
    \end{equation*}
    and Bouc's Theorem~\ref{lemma:boucEI2} shows that the inclusion of
    $\catt F{}{*+\mathrm{eab}}$ into $\catt F{}*$ is a homotopy
    equivalence.

\end{itemize}
\end{proof}

We now turn to the question of finding a ``large'' collection of subgroups that controls  the homotopy type of the exterior quotient, in some sense dual to the elementary abelian subgroups of Theorem~\ref{thm:OrbitCatElemAb}.  There is a new technical difficult we must take into consideration here:  We lack a good understanding of the full exterior quotient of a Frobenius $P$-category.  Much more is known about the $\cat F{}{}$-selfcentralizing subcategory $\catt F{}{\mathrm{sfc}}$, where we can make use of the following facts:

  \begin{itemize}
  \item All \m s in $\catt F{}{\mathrm{sfc}}$ are epi\m s
    \cite[Corollary 4.9]{puig09}, therefore

  \item  The categories $H/\catt F{}{\mathrm{sfc}}$ are thin

  \item The \we\ for $\catt F{G}{\mathrm{sfc}}$ vanishes off the $\cat
    F{G}{}$-radical subgroups \cite[Corollary 8.6]{jmm_mwj:2010}

% \item  $\cat F{}*$ and $\catt F{}*$ have identical co\we s  [SOURCE?]
  \end{itemize}
  
  We will also need the following technical result, which is a
  reformulation of \cite[Proposition 2.4]{diaz_libman09}:
    
      \begin{lemma}\label{lemma:diaz_libman}
    Let $H$, $N$, and $K$ be objects of $\cat F{}{}$ such that $H$ is
    $\cat F{}{}$-selfcentralizing and $H \leq N \leq N_P(H)$.  An
    $\cat F{}{}$-\m\ $\func \varphi HK$ extends to an $\cat F{}{}$-\m\
    $\func \psi NK$ if and only if $\cat FN{}(H)^\varphi \leq \cat
    FK{}(H^\varphi)$.
  \end{lemma}
  \begin{proof}
    We prove the ``if'' implication, as the converse is clear.  Since $H$ is $\cat F{}{}$-selfcentralizing, the same is true of
    $H^\varphi$ and thus $H^\varphi$ is fully centralized in $\cat
    F{}{}$ \cite[4.8]{puig09}. By the Extension Axiom for Frobenius
    $P$-categories and our assumption, $\func \varphi HK$ extends to a \m\ $\func \rho
    NP$ \cite[2.10.1]{puig09}. We claim that $(x)\rho \in K$ for all
    $x \in N$. By assumption, there is some $y \in K$ such that
    conjugation with $(x)\rho$ and with $y$ has the same effect on
    $H^\varphi$. This means that $(x)\rho y^{-1} \in C_P(H^\varphi)
    \leq Z(H^\varphi) \leq H^\varphi \leq K$, and thus $(x)\rho \in
    K$.  The corestriction $\psi = K \vert \rho \colon N \to K$ of
    $\func \psi NP$ extends $\func \varphi HK$.
  \end{proof}

  Consequently,
  \begin{equation*}
    \catt F{}{}(N,K) = \catt F{}{}(H,K)^{\cat FN{}(H)}
  \end{equation*}
  under the assumptions of Lemma~\ref{lemma:diaz_libman}.

%  \begin{proof}[Proof of Corollary~\ref{cor:FtotildeF}] 
%    This corollary follows immediately from the commutative diagram
%    \begin{equation*}
%      \xymatrix{
%        {\cat F{}*} \ar[r] & {\catt F{}*} \\
%        {\cat F{}{*+\mathrm{eab}}}  \ar@{^(->}[u]^{\simeq} \ar@{=}[r] &
%        {\catt F{}{*+\mathrm{eab}}} \ar@{_(->}[u]_{\simeq}}
%    \end{equation*}
%   %  $\xymatrix{{\catt F{}{*}} & 
%%       {\catt F{}{\mathrm{eab}}} \ar@{_(->}[l]_{\simeq} \ar@{=}[r] &
%%       {\cat F{}{\mathrm{eab}}} \ar@{^(->}[r]^{\simeq} &
%%       {\cat F{}{*}} }$
%    where the vertical \m s are homotopy equivalences.
%  \end{proof}

We are now ready to prove: 

  \begin{thm}
    The  inclusion $ \catt F{}{\mathrm{sfc}+\mathrm{rad}}\hookrightarrow \catt F{}{\mathrm{sfc}}$ is a homotopy equivalence.
  \end{thm}

  \begin{proof}  Fix an $\cat F{}{}$-selfcentralizing subgroup $H\leq P$, and let $G:=\catt F{}{}(H)$ be the automorphism group of $H$ in the exterior quotient category.
  \begin{itemize}
  \item[{\bf[EC]}]
    Consider the special case that $\catt F {}{\mathrm{sfc}}=\catt F G{\mathrm{sfc}}$ for some finite group $G$ inducing the exterior quotient category $\catt F{}{}$.  The weighting for $\catt F G{\mathrm{sfc}}$ was computed in \cite[Proposition 8.5]{jmm_mwj:2010}; comparison with Theorem~\ref{lemma:underwt} yields
     \begin{equation*}
       \frac{-\rchi(\cat S{\cattc FG(H)}*)}{|\cattc FG(H)|} =
        k^{[H]}_{[\cattc FG]} = 
        \frac{-\rchi(H//\cattc FG)}{|\cattc FG(H)|}.
     \end{equation*}
     Thus the proper coslice category $H//\cattc FG$ and the poset
     $\cat S{\cattc FG(H)}*$ have identical Euler characteristics for any object $H$
     of $\catt F{}{\mathrm{sfc}}$.  
     
     We believe that an abstract version of \cite[Proposition 8.5]{jmm_mwj:2010} that does not reference an ambient finite group is true as well.  Such a result would imply that $H//\catt F{}{\mathrm{sfc}}$ and $\cat S{\catt F{}{\mathrm{sfc}}(H)}*$ should in general have identical reduced Euler characteristics.  (In fact, this result will follow from the homotopy equivalence of the next paragraph.)  As the Euler characteristic computation serves primarily to direct our attention toward the class of subgroups that control the homotopy type, this special case is already enough to suggest that we should consider the $\cat F{}{}$-radical subgroups.  That is where we will focus our attention.
     
     \item[{\bf[HE]}] We claim there is a homotopy equivalence $H//\catt F{}{}\simeq\cat S{\catt F{}{}(H)}*$. There are functors
   \begin{equation} \label{eq:rHtildeF}
      r_H \colon H/\catt F{}{} \to \cat S{\catt F{}{}(H)}{}, \qquad
      r_H \colon H//\catt F{}{} \to \cat S{\catt F{}{}(H)}{*}
    \end{equation}
    There is no loss of generality in  assuming that $H$ is fully
    normalized in $\catt F{}{}$, so that the order of the $P$-normalizer of $H$ is maximal in its $\catt F{}{}$-isomorphism class.  The functor $r_H$ is defined by $[\varphi]r_H~=~{}^\varphi\catt FK{}(H^\varphi)$, where
    $[\varphi]=\varphi \cat FK{}{}(K) \in \catt F{}{}(H,K) = \cat
    F{}{}(H,K)/\cat FK{}(K)$ is an object of $ H/\catt F{}{}$.  Note that this is well-defined
    even though $\varphi$ is only defined up to conjugacy in $K$.  The
    group
     \begin{equation*}
    \catt FK{}(H^\varphi) = 
    C_K(H^\varphi) \backslash N_K(H^\varphi) / H^\varphi =
    Z(H^\varphi) \backslash N_K(H^\varphi) / H^\varphi =
    N_K(H^\varphi) / H^\varphi
  \end{equation*}
  and the isomorphic group $r_H(\varphi \cat FK{}(K)) =  
  {}^\varphi\catt FK{}(H^\varphi)$ are related by the commutative diagram
    \begin{equation*}
      \xymatrix{
        H \ar@{.>}[d]_{{}^\varphi\catt FK{}(H^\varphi)} 
        \ar[r]^\varphi_\cong &
        H^\varphi  \ar[d]^{\catt FK{}(H^\varphi) = N_K(H^\varphi)/H^\varphi} \\
        H \ar[r]^\varphi_\cong & H^\varphi}
    \end{equation*}
    It is clear that ${}^\varphi\catt FK{}(H^\varphi)$ is  a
    $p$-subgroup of $\catt F{}{}(H)$ and that $r_H(\varphi_1) \leq
    r_H(\varphi_2)$ whenever there is a $\catt F{}{}$-\m\
    \begin{equation*}
      \xymatrix{
        & H \ar[dl]_{\varphi_1 \cat F{K_1}{}(K_1)} 
        \ar[dr]^{\varphi_2 \cat F{K_2}{}(K_2)} \\
        K_1 \ar[rr] && K_2 }
    \end{equation*}
    under $H$. Thus $r_H$ is a functor.  We now want to use Quillen's
    Theorem A to show that $r_H$ is an equivalence of categories.

    Let $\overline{L}$ be a $p$-subgroup of $\catt F{}{}(H) = \cat
    F{}{}(H)/\cat FH{}(H)$. We may assume that $\overline{L}$ is
    contained in the \syl p\ $\catt FP{}(H) = N_P(H)/H$ of $\catt
    F{}{}(H)$ (which is known to be Sylow by the assumption that $H$ is fully normalized in $\catt F{}{}$). There is a unique $p$-subgroup $L \in [P,N_P(H)]$ such
    that $\overline{L} = L/H$.  The category $\overline{L}/r_H$ is the
    full subcategory of $H/\catt F{}{}$ generated by all objects
    $\varphi \cat FK{}(K) \in \catt F{}{}(H,K)$ such that
    $\overline{L}^\varphi \leq \catt FK{}(H^\varphi) =
    N_K(H^\varphi)/H^\varphi$, or, equivalently, $L^\varphi \leq
    N_K(H^\varphi)$. Here is an attempt to visualize this relationship:
  \begin{equation*}
    \xymatrix{
      {H} 
      \ar[d]_{L/H} 
        \ar[r]_-\cong^-\varphi &
        {H^\varphi}  
        \ar[d]^{N_K(H^\varphi)/H^\varphi} \\
         {H} \ar[r]_-\cong^-\varphi & {H^\varphi}}
  \end{equation*}
  The inclusion $\iota^H_L \colon H
  \hookrightarrow L$ of $H$ into $L$ represents both a \m\ in $\catt
  F{}{}(H,L)$ and an object of $\overline{L}/r_H$ because
  $\overline{L}$ is contained in $(\iota^H_L\cat FL{}(L))r_H =
  N_L(H)/H = L/H = \overline{L}$. By Lemma~\ref{lemma:diaz_libman}
  there is an extension in $\catt F{}{}$ of $\varphi \colon H \to K$
  \begin{equation*}
    \xymatrix{
      & H \ar@{_(->}[dl]_{\iota^H_L} \ar[dr]^{\varphi} \\
      L \ar@{.>}[rr] && K}
  \end{equation*}
  to a \m\ $L \to K$.
% %%%%%%%%%%%
%  that
%   \begin{equation*}
%     \catt F{}{}(L,K) \to \catt F{}{}(H,K)^{\overline{L}}
%   \end{equation*}
%   is a bijection. 
%%%%%%%%%%%%%%
  We have now shown that $\iota^H_L\cat FL{}(L)$ is an initial object
  of $\overline{L}/r_H$ for any object $\overline{L}$ of $\cat S{\catt
    F{}{}(H)}{}$. According to Quillen's Theorem A
  (Theorem~\ref{thm:QuillenA}), $r_H$ is a homotopy equivalence of
  categories.

 %  \mynote{Alternative}: We will define a functor in the opposite
%   direction - unfortunately, this functor is only defined on subgroups
%   of the subgroup $\catt FS{}(H)$ of $\catt F{}{}(H)$. Any
%   $p$-subgroup of $\catt FS{}(H) = C_S(H) \backslash N_S(H) / H$ has
%   the form $C_S(H) \backslash K / H$ for a unique subgroup $K \leq
%   N_S(H)$ containing $HC_S(H)$. Set
%   \begin{equation*}
%     i_H( C_S(H) \backslash K / H) = (H \leq K)
%   \end{equation*}
%   \mynote{Now what?}
  
  Since the functor $r_H$ takes noniso\m s $\varphi\cat FK{}(K) \in
  \catt F{}{}(H,K) \subset \Ob{H/\catt F{}{}}$ to nonidentity
  $p$-subgroups of $\catt F{}{}(H)$, it restricts to a functor $r_H
  \colon H//\catt F{}{} \to \cat S{\catt F{}{}(H)}*$ of $H//\catt
  F{}{}$ into the Brown poset of the auto\m\ group of $H$. But since
  $\overline{L}/r_H$ is contractible for any nonidentity
  $p$-subgroup $\overline{L}$ of $\catt F{}{}(H)$, we already know
  that the restricted functor $r_H$ is a homotopy equivalence of
  categories. By property \eqref{eq:Frad},
    \begin{equation*}
      \mathrm{supp}(\bullet//\catt F{}{\mathrm{sfc}}) \subset
          \Ob{\catt F{}{\mathrm{sfc}+\mathrm{rad}}}
    \end{equation*}
    and Bouc's Theorem~\ref{lemma:boucEI2} shows that the inclusion of
    $\catt F{}{\mathrm{sfc}+\mathrm{rad}}$ into $\catt
    F{}{\mathrm{sfc}}$ is a homotopy equivalence.
\end{itemize}

  \end{proof}

  \section{Linking categories}
  \label{sec:link}

  Let $\cat L{}{\mathrm{sfc}}$ be the centric linking system
  associated to a Frobenius $P$-category $\cat F{}{}$ \cite[Definition
  1.7]{blo03}.  We will prove that the homotopy type of $\cat L{}{\mathrm{sfc}}$ is controlled by the $\cat F{}{}$-radical subgroups.  This result is part of \cite[Theorem B]{bcglo01}, the full strength of which would be accessible by our methods if we were to consider the more general notion of a \emph{quasi}centric linking system.

%  \mynote{Maybe $\cat L{}{}$ can be any divisible
%     \cite[17.7.1]{puig09} or perfect \cite[17.13]{puig09} $\cat
%     F{}{}$-locality? But how many perfect $\cat F{}{}$-localities are
%     there? (Chermak)}

We will need the following facts:
  \begin{itemize}
  \item All \m s in $\cat L{}{\mathrm{sfc}}$ are mono\m s and epi\m s
     \cite[Proposition 24.2]{puig09} %% Oliver--Ventura?
 % so that categories $H$ over $\cat
 % L{}{\mathrm{sfc}}$ are thin.
  \item The \we\ for $\cat L{}{\mathrm{sfc}}$ vanishes off the $\cat
    F{}{}$-radical subgroups \cite[Proposition 8.5]{jmm_mwj:2010}
  \end{itemize}

  % Let $\catt L{G}{\mathrm{sfc}+\mathrm{rad}}$ be the full subcategory
%   of $\catt L{G}{\mathrm{sfc}}$ generated by the $\cat F{G}{}$-radical
%   $p$-selfcentralizing $p$-subgroups of $G$.

  \begin{thm}
    The inclusion functor $\cat L{}{\mathrm{sfc}+\mathrm{rad}} \to \cat
    L{}{\mathrm{sfc}}$ is a homotopy equivalence.
  \end{thm}
  \begin{proof}
    Let $H$ be a $\cat F{}{}$-selfcentralizing object of $\cat F{}{}$.
    The functor $\widetilde{\pi} \colon \cat L{}{\mathrm{sfc}} \to
    \catt F{}{\mathrm{sfc}}$ is bijective on objects and $|K|$-to-$1$
    for on \m\ sets $\cat L{}{\mathrm{sfc}}(H,K) \to \catt
    F{}{\mathrm{sf c}}(H,K)$ with codomain $K \in \Ob{\cat
      F{}{\mathrm{sfc}}}$.  $K = \cat LK{}(K) \leq \cat L{}{}(K)$ acts
    freely from the right on $\cat L{}{}(H,K)$ with quotient $\cat
    L{}{\mathrm{sfc}}(H,K)/K = \catt F{}{\mathrm{sfc}}(H,K)$
    \cite[Lemma 1.10]{blo03}.  This implies that if $\varphi_1 \in
    \cat L{}{}(H,K_1)$, $\varphi_2 \in \cat L{}{}(H,K_2)$, and the
    commutative $\catt F{}{}$-diagram to the right has a solution
    \begin{equation*}
      \xymatrix{&H \ar[dl]_-{\varphi_1} \ar[dr]^-{\varphi_2}
       \ar@{}[d]|-*+<8pt>[Fo]{\cat L{}{}} \\
        K_1 \ar@{.>}[rr] && K_2} \qquad\qquad
      \xymatrix{&H \ar[dl]_{(\varphi_1)\pi} \ar[dr]^-{(\varphi_2)\pi} 
        \ar@{}[d]|-*+<8pt>[Fo]{\catt F{}{}} \\
        K_1 \ar@{.>}[rr] && K_2}
    \end{equation*}
    then the commutative $\cat L{}{}$-diagram to the left has a unique
    solution \cite[Lemma 1.10]{blo03}.  Consider the functor
    \begin{equation*}
      H/\widetilde{\pi} \colon H/\cat L{}{\mathrm{sfc}} \to 
      H/\catt F{}{\mathrm{sfc}}
    \end{equation*}
    induced by the functor $\widetilde{\pi} \colon \cat
    L{}{\mathrm{sfc}} \to \catt F{}{\mathrm{sfc}}$.  The above
    considerations mean that any $\varphi \in \cat L{}{}(H,K) \subset
    \Ob{H/\cat L{}{}}$ is initial in the category
    $(\varphi)\widetilde{\pi}/H/\widetilde{\pi}$. By Quillen's Theorem
    A (Theorem~\ref{thm:QuillenA}), $H/\widetilde{\pi}$ is a homotopy
    equivalence.

    Restricting to the noniso\m s we get a homotopy equivalence
    $H//\widetilde{\pi} \colon H//\cat L{}{} \to H//\catt F{}{}$.
    Compose these homotopy equivalences with the homotopy equivalences
    of \eqref{eq:rHtildeF} to get homotopy equivalences
    \begin{equation} \label{eq:rHL}
       H/\cat L{}{\mathrm{sfc}} \to \cat S{\catt F{}{}(H)}{}, \qquad
       H//\cat L{}{\mathrm{sfc}} \to \cat S{\catt F{}{}(H)}*
    \end{equation}
    By property \eqref{eq:Frad},
    \begin{equation*}
      \mathrm{supp}(\bullet//\cat L{}{\mathrm{sfc}}) \subset
          \Ob{\cat L{}{\mathrm{sfc}+\mathrm{rad}}}
    \end{equation*}
    and Bouc's Theorem~\ref{lemma:boucEI2} shows that the inclusion of
    $\cat L{}{\mathrm{sfc}+\mathrm{rad}}$ into $\cat
    L{}{\mathrm{sfc}}$ is a homotopy equivalence.
  \end{proof}
%%%%%%%%%%%%%%%%%%%%%
%   \begin{proof}
%     For any object $H$ of $\cat LG{\mathrm{scf}}$ we have from
%     \cite[Proposition 8.5]{jmm_mwj:2010} and
%     Theorem~\ref{lemma:underwt} that the \we\
%      \begin{equation*}
%        \frac{-\rchi(\cat S{\cattc FG(H)}*)}{|\cat LG{\mathrm{sfc}}(H)|} =
%        k^{[H]}_{[\cat LG{\mathrm{sfc}}]} = 
%        \frac{-\rchi(H//\cat LG{\mathrm{sfc}})}{|\cat LG{\mathrm{sfc}}(H)|}
%      \end{equation*}
%      so that the strict under-category $H//\cat LG{}$ and
%      the poset $\cat S{\catt FG{}(H)}*$ have identical \Euc .
%      \mynote{Are they homotopy equivalent?}

%      When $H$ is $\cat F{}{}$-selfcentralizing 
%      \begin{equation*}
%        H/\cat LG{} \to H/\catt FG{} 
%      \end{equation*}
%      should be a homotopy equivalence.

%      When $H$ is $G$-selfcentralizing we have that $C_G(P) = O^pC_G(H)
%      \times Z(H)$ and 
%      \begin{multline*}
%        \catt FG{}(H,K) = C_G(H) \backslash N_G(H,K) / K
%                        = (O^pC_G(H)\times Z(H)) \backslash N_G(H,K)/K \\
%                        = O^pC_G(H) \backslash N_G(H,K)/K
%                        = \cat LG{}(H,K)/K
%      \end{multline*}
%      Here we use that $H N_G(H,k) \subset N_G(H,K)K$ because $hg =
%      gh^g$ when $h \in H$ and $g \in N_G(H,K)$. 
%   \end{proof}
%%%%%%%%%%%%%%%%%%%%%%%%%%

  It is worth noting that the main connection between the theory of Frobenius $P$-categories and topology comes from the classifying space of $\cat L{}{\mathrm{sfc}}$, which should be thought of as a generalization of the $p$-completion of the classifying space of a finite group.  What is interesting in the preceding proof  is that we are able to show that control of homotopy for the linking system actually comes from the seemingly less natural question about control of homotopy in the exterior quotient category $\catt F{}{\mathrm{sfc}}$.  

Finally, we close with the dual statement, where the homotopy type is controlled by the ``small'' nonidentity elementary abelian groups subgroups of $P$.  As there is currently no clear definition for an abstract linking system which has all nonidentity subgroups of $P$ as objects, we will restrict our attention to the case where an actual finite group induces $\cat LG*$.

  \begin{prop}\label{prop:LG*eab}
     The inclusion $\cat L{G}{*+\mathrm{eab}} \to \cat L{G}{*}$  is
  a homotopy equivalence.
  \end{prop}
  \begin{proof}
    The two expressions, from \cite[Theorem 1.1.(2)]{jmm_mwj:2010},
    and Lemma~\ref{lemma:(1,P)contract}.\eqref{lemma:(1,P)contractB}
    and Theorem~\ref{lemma:underwt}, for the co\we\ for $[\cat L{G}*]$
    \begin{equation*}
       \frac{-\rchi(\cat SK{(1,K)})}{|\cat L{G}*(K)|} =
        k_{[K]}^{[\cat L{G}*]} = 
        \frac{-\rchi(\cat L{G}*//K)}{|\cat L{G}*(K)|}
    \end{equation*}
    show that $\cat SK{(1,K)}$ and $\cat L{G}*//K$ have identical \Euc
    s for any object $K$ of $\cat L{G}*$. In fact they are homotopy
    equivalent as we see in much the same way as in the proof of
    Theorem~\ref{prop:FeabtoF*II}. The proof now follows from
    Bouc's Theorem~\ref{lemma:boucEI2} because $\mathrm{supp}(\cat
    LG*//K) \subset \Ob{\cat LG{*+\mathrm{eab}}}$ by
    Lemma~\ref{lemma:(1,P)contract}.\eqref{lemma:(1,P)contractA}.
  \end{proof}

  % \begin{rmk}\label{rmk:LG}
%     Suppose that $G$ is a finite group and let $\cat LG{}$ and $\catt
%     FG{}$ the associated $p$-subgroup categories.  When $H$ is
%     $G$-selfcentralizing we have that $C_G(P) = O^pC_G(H) \times Z(H)$
%     and
%      \begin{multline*}
%        \catt FG{}(H,K) = C_G(H) \backslash N_G(H,K) / K
%                        = (O^pC_G(H)\times Z(H)) \backslash N_G(H,K)/K \\
%                        = O^pC_G(H) \backslash N_G(H,K)/K
%                        = \cat LG{}(H,K)/K
%      \end{multline*}
%      Here we use that $H N_G(H,k) \subset N_G(H,K)K$ because $hg =
%      gh^g$ when $h \in H$ and $g \in N_G(H,K)$. By a miracle, the
%      action of $K$ on $\cat L{}{}(H,K)$ is free.
%   \end{rmk}

%\bibliographystyle{amsplain}
%\bibliography{top}

\begin{thebibliography}{10}

\bibitem{bouc84a}
Serge Bouc, \emph{Homologie de certains ensembles ordonn\'es}, C. R. Acad. Sci.
  Paris S\'er. I Math. \textbf{299} (1984), no.~2, 49--52. \MR{756517
  (85k:20150)}

\bibitem{bk}
A.K. Bousfield and D.M. Kan, \emph{Homotopy limits, completions and
  localizations}, 2nd ed., Lecture Notes in Mathematics, vol. 304,
  Springer-Verlag, Berlin--Heidelberg--New York--London--Paris--Tokyo, 1987.

\bibitem{bcglo01}
Carles Broto, Nat{\`a}lia Castellana, Jesper Grodal, Ran Levi, and Bob Oliver,
  \emph{Subgroup families controlling {$p$}-local finite groups}, Proc. London
  Math. Soc. (3) \textbf{91} (2005), no.~2, 325--354. \MR{MR2167090
  (2007e:20111)}

\bibitem{blo1}
Carles Broto, Ran Levi, and Bob Oliver, \emph{Homotopy equivalences of
  {$p$}-completed classifying spaces of finite groups}, Invent. Math.
  \textbf{151} (2003), no.~3, 611--664. \MR{MR1961340 (2004c:55031)}

\bibitem{blo03}
\bysame, \emph{The homotopy theory of fusion systems}, J. Amer. Math. Soc.
  \textbf{16} (2003), no.~4, 779--856 (electronic). \MR{1 992 826}

\bibitem{brown75}
Kenneth~S. Brown, \emph{Euler characteristics of groups: the {$p$}-fractional
  part}, Invent. Math. \textbf{29} (1975), no.~1, 1--5. \MR{0385008 (52
  \#5878)}

\bibitem{diaz_libman09}
Antonio D{\'{\i}}az and Assaf Libman, \emph{The {B}urnside ring of fusion
  systems}, Adv. Math. \textbf{222} (2009), no.~6, 1943--1963. \MR{2562769
  (2011a:20039)}

\bibitem{dwyer97}
W.~G. Dwyer, \emph{Homology decompositions for classifying spaces of finite
  groups}, Topology \textbf{36} (1997), no.~4, 783--804. \MR{1432421
  (97m:55016)}

\bibitem{dw:codec}
W.~G. Dwyer and C.~W. Wilkerson, \emph{A cohomology decomposition theorem},
  Topology \textbf{31} (1992), no.~2, 433--443. \MR{MR1167181 (93h:55008)}

\bibitem{gabriel_zisman67}
P.~Gabriel and M.~Zisman, \emph{Calculus of fractions and homotopy theory},
  Ergebnisse der Mathematik und ihrer Grenzgebiete, Band 35, Springer-Verlag
  New York, Inc., New York, 1967. \MR{0210125 (35 \#1019)}

\bibitem{gorenstein68}
Daniel Gorenstein, \emph{Finite groups}, Harper \& Row Publishers, New York,
  1968. \MR{MR0231903 (38 \#229)}

\bibitem{jackowski_slominska2001}
Stefan Jackowski and Jolanta S{\l}omi{\'n}ska, \emph{{$G$}-functors,
  {$G$}-posets and homotopy decompositions of {$G$}-spaces}, Fund. Math.
  \textbf{169} (2001), no.~3, 249--287. \MR{1852128 (2002h:55017)}

\bibitem{jmm_mwj:2010}
Martin~Wedel Jacobsen and Jesper~M. M{\o}ller, \emph{Euler characteristics and
  {M}\"obius algebras of {$p$}-subgroup categories}, J. Pure Appl. Algebra
  \textbf{216} (2012), no.~12, 2665--2696. \MR{2943749}

\bibitem{leinster08}
Tom Leinster, \emph{The {E}uler characteristic of a category}, Doc. Math.
  \textbf{13} (2008), 21--49,
  \href{http://www.math.uiuc.edu/documenta/vol-13/02.pdf}{Doc. Math.}
  \MR{MR2393085}

\bibitem{maclane}
Saunders MacLane, \emph{Homology}, first ed., Springer-Verlag, Berlin, 1967,
  Die Grundlehren der mathematischen Wissenschaften, Band 114. \MR{MR0349792
  (50 \#2285)}

\bibitem{puig09}
Llu{\'{\i}}s Puig, \emph{Frobenius categories versus {B}rauer blocks}, Progress
  in Mathematics, vol. 274, Birkh\"auser Verlag, Basel, 2009, The Grothendieck
  group of the Frobenius category of a Brauer block. \MR{MR2502803}

\bibitem{quillen73}
Daniel Quillen, \emph{Higher algebraic {$K$}-theory. {I}}, Algebraic
  {$K$}-theory, {I}: {H}igher {$K$}-theories ({P}roc. {C}onf., {B}attelle
  {M}emorial {I}nst., {S}eattle, {W}ash., 1972), Springer, Berlin, 1973,
  pp.~85--147. Lecture Notes in Math., Vol. 341. \MR{0338129 (49 \#2895)}

\bibitem{quillen78}
\bysame, \emph{Homotopy properties of the poset of nontrivial {$p$}-subgroups
  of a group}, Adv. in Math. \textbf{28} (1978), no.~2, 101--128. \MR{MR493916
  (80k:20049)}

\bibitem{robinson:groups}
Derek J.~S. Robinson, \emph{A course in the theory of groups}, second ed.,
  Springer-Verlag, New York, 1996. \MR{96f:20001}

\bibitem{stanley97}
Richard~P. Stanley, \emph{Enumerative combinatorics. {V}ol. 1}, Cambridge
  Studies in Advanced Mathematics, vol.~49, Cambridge University Press,
  Cambridge, 1997, With a foreword by Gian-Carlo Rota, Corrected reprint of the
  1986 original. \MR{MR1442260 (98a:05001)}

\bibitem{thomason79}
R.~W. Thomason, \emph{Homotopy colimits in the category of small categories},
  Math. Proc. Cambridge Philos. Soc. \textbf{85} (1979), no.~1, 91--109.
  \MR{510404 (80b:18015)}

\end{thebibliography}
%\input section10

\def\cprime{$'$} \def\cprime{$'$} \def\cprime{$'$} \def\cprime{$'$}
\providecommand{\bysame}{\leavevmode\hbox to3em{\hrulefill}\thinspace}
\providecommand{\MR}{\relax\ifhmode\unskip\space\fi MR }
\providecommand{\MRhref}[2]{%
  \href{http://www.ams.org/mathscinet-getitem?mr=#1}{#2}
}
\providecommand{\href}[2]{#2}

\end{document}